\documentclass{amsart}
\usepackage{amsmath, amssymb, mathrsfs, verbatim, multirow}

\usepackage{amsthm}
\usepackage{xcolor}

\usepackage{pgfplots}
\pgfplotsset{compat=1.15}
\usetikzlibrary{arrows}

\usepackage[utf8]{inputenc}
\usepackage{tikz}
\usepackage{tikz-cd}
\usepackage{subcaption} 
\usepackage{graphicx}

\theoremstyle{plain}

\newtheorem{Teo}{Theorem}[section]
\newtheorem{Def}[Teo]{Definition}
\newtheorem{Prop}[Teo]{Proposition}
\newtheorem{Obs}[Teo]{Remark}
\newtheorem{Lema}[Teo]{Lemma}
\newtheorem{Cor}[Teo]{Corollary}

\newtheorem{Exa}[Teo]{Example}

\newtheorem{Result}[Teo]{Main Result}
\newcommand{\Q}{\mathbb{Q}}
\newcommand{\R}{\mathbb{R}}
\newcommand{\Z}{\mathbb{Z}}
\newcommand{\N}{\mathbb{N}}

\newcommand{\Llr}{\Longleftrightarrow}

\newcommand{\VR}{\mathcal{O}}

\newcommand{\MI}{\mathfrak{m}}

\newcommand{\SU}{\mbox{supp}}
\newcommand{\supp}{\mbox{\rm supp}}

\begin{document}
	\title[On topologies on the space of valuations  and the valuative tree]{On topologies on the space of valuations\\ and the valuative tree}
	\author{V. Manfredini, J. Novacoski and C. H. Silva de Souza}
	\thanks{During the realization of this project the authors were supported by a grant from Funda\c c\~ao de Amparo \`a Pesquisa do Estado de S\~ao Paulo (process numbers 2017/17835-9 and 2021/13531-0) and from CAPES (process number 88887.841249/2023-00).}

\date{}
\begin{abstract} In this paper, we discuss topological aspects of the space of valuations $\mathbb{V}$ and the valuative tree $\mathcal{T}(v,\Lambda)$ in the sense of \cite{Nart_1}. We present a relation between the weak tree topology and the Scott topology in $\mathcal{T}(v,\Lambda)$ and describe the supremum of an increasing family of valuations in a special subtree. We also view the valuative tree as a subset of the product $(\Lambda_\infty)^{K[x]}$ and prove that it is closed if we consider the natural product topology.  

\end{abstract}


\address{Department of Mathematics\\
Federal University of São Carlos,\newline \indent
São Carlos-SP,
Brazil}
\email{vmmanfredini@estudante.ufscar.br}

\email{josnei@ufscar.br}

\email{caiohss@estudante.ufscar.br}

\maketitle

\section{Introduction}



Topologies on spaces of valuations appear in many different contexts. The most common is the \textit{Zariski topology}. It was introduced by Zariski in the first half of the twentieth century and it has been extensively studied since then. Initially, it was defined as a topology on algebraic varieties, but in a modern language it is defined as a topology on the spectrum of a ring $R$, i.e., the set of all prime ideals of $R$. The space of \textit{Krull valuations} on a ring (or valuations on a field) admits a natural structure as inverse limit of spectra of rings with their respective Zariski topologies. The corresponding topology is called again the Zariski topology on the space of valuations.

The Zariski topology is constructed in the set of Krull valuations of a ring and its definition cannot be extended directly to the set of all valuations on this ring. To overcome this problem, Huber and Knebusch introduced the \textit{valuation spectrum topology} (see \cite{Kne}).

The valuative tree structure also appears as an alternative in the study of sets of valuations. It was developed by Favre and Jonsson in \cite{Fav_1}, where it is connected to the theory of \textit{Berkovich spaces}. In this original context, the valuative tree consisted of all (normalized and centered) valuations in $\mathbb{C}[[x,y]]$, the ring   of formal power series in two variables over the field of complex numbers, and taking values in $\R$. Favre and Jonsson give a (rooted non-metric) tree structure to this space. This tree structure induces different topologies on this set.

More recently, this concept was generalized in \cite{Nart_1}. Namely, one can consider the  set of all valuations in a polynomial ring $K[x]$, taking values in a given totally ordered abelian group,  extending a fixed valuation on $K$.

From now on, we fix a valued field $(K,v)$ and a totally ordered abelian group $\Lambda$, containing the value group of $v$. Denote by $\mathbb{V}$ the set of all equivalence classes of valuations in the polynomial ring in one variable $K[x]$ whose restriction to $K$ is equivalent to $v$. Also, denote by $\mathcal{T}=\mathcal{T}(v,\Lambda)$ set of all valuations on $K[x]$, taking values in $\Lambda$,  whose restriction  to $K$ is  $v$. We will refer to $\mathbb{V}$ as the \textbf{space of valuations} and to $\mathcal T$ as the \textbf{valuative tree}. In this paper, we discuss topological aspects of $\mathbb{V}$ and $\mathcal T$.

Every tree $(\mathcal{T},\leq)$ is, by definition, a partially ordered set. Hence, we can consider the \textit{Scott topology} on it. Another natural topology  on $\mathcal{T}$ is the so called \textit{weak tree topology}, which is defined using the notion of intervals in $\mathcal{T}$.  
Let $\mathfrak{S}_{\mu}$ be the Scott topology defined over the rooted tree $(\mathcal{T}, \leq_\mu)$, for $\mu\in \mathcal{T}(v, \Lambda)$ (see Section 4 for the respective definitions). 

We will say that $\Lambda$ is \textbf{densely ordered} if for all $\gamma, \lambda \in \Lambda$ with $\gamma < \lambda$, there exists $\alpha \in \Lambda$ such that $\gamma < \alpha < \lambda$. Our first main result is the following.

\begin{Result}(Theorem \ref{teoWeakTreeTopGenByScott}) Assume that $\Lambda$ is densely ordered. Then the weak tree topology coincides with the topology generated by $\bigcup_{\mu\in\mathcal{T}}\mathfrak{S}_{\mu}$.
    
\end{Result}

In \cite{Nart_1}, a universal ordered abelian group $\Lambda$ is presented having the following property. Each class of valuations whose restriction to $K$ is equivalent to $v$ contains a valuation $\nu:K[x]\longrightarrow \Lambda_\infty$. Then, it is possible to construct a subtree $\mathcal{T}_{\rm sme}$ which is in bijection with the space of valuations $\mathbb{V}$ (see Section  \ref{SecTreeSME} for details). 

In recent works, \textit{increasing families of valuations} appear as important tools for understanding the valuative tree (see \cite{Nart_1}, \cite{Nart_2}, \cite{CaioJosneiparametrizations} and \cite{josneicaio2}). For instance, each valuation $\nu\in \mathcal{T}$ can be achieved 
 by a process of successive augmentations that results in what is called  a \textit{Mac Lane-Vaquié chain}, which is a special type of increasing family of valuations (see \cite{Nart_2}). Our second main result deals with the notion of supremum of a  general increasing family in $\mathcal{T}_{\rm sme}$.

\begin{Result}(Theorem \ref{teoSupinTsme}) Every increasing family of valuations in $\mathcal{T}_{\rm sme}$ admits a supremum.
    
\end{Result}

Another topology on $\mathcal{T}$, that appears naturally, is obtained when we consider a topology in $\Lambda_\infty$ and look at $\mathcal{T}$  as a subspace of the product $(\Lambda_\infty)^{K[x]}$, endowed with the product topology.  More generally, we can study the set $\mathcal{V}_{R}(\Lambda)$ of all valuations on a ring $R$ taking values in $\Lambda_{\infty}$ with the induced product topology. Our third main result gives a criterion to decide when $\mathcal{V}_{R}(\Lambda)$ is closed in $(\Lambda_\infty)^{R}$.

\begin{Result}(Theorem \ref{criterio}) Let $\Lambda'$ be any submonoid of $\Lambda_{\infty}$ and take a topology $\mathfrak{U}$ on $\Lambda'$ such that
	\begin{description}
		\item [(P1)] the addition $+ : \Lambda'\times\Lambda'\longrightarrow\Lambda'$ is continuous, and
		\item [(P2)] for every $\gamma, \gamma' \in \Lambda'$ such that $\gamma<\gamma'$ there exist open sets $U, U' \in \mathfrak{U}$ such that $\gamma \in U, \gamma' \in U'$ and $U < U'$ (i.e., $u<u'$ for every $u \in U$ and $u'\in U$).
	\end{description}
	Then the set $\mathcal{V}_{R}(\Lambda')$ of valuations of $R$ taking values in $\Lambda'$ is closed in $(\Lambda')^R$.
\end{Result}

We now describe the structure of this paper. In Section \ref{SecPreliminaries}, we compile some definitions and results that are the basis of our study. In Section \ref{SecSpaceOfValuation}, we define $\mathbb{V}$ and study the valuation spectrum topology on it. We prove that the subsets $\mathbb{V}_{\rm vt}$ (of classes of \textit{value-transcendental} valuations) and $\mathbb{V}_{\rm rt}$ (of classes of \textit{residue-transcendental} valuations) are dense in $\mathbb{V}$ when equipped with the valuation spectrum topology (see Proposition \ref{PropVTRTDenseValSpecTop}).

In Section \ref{SecLambdaTree}, we explore connections between two different topologies on the valuative tree.  We study the weak tree topology and its relationship with the theory of key polynomials and tangent directions found in \cite{Nart_1} (see Proposition \ref{carc}).  Then, we define the Scott topology in the valuative tree and see that it is strictly coarser than the weak tree topology (see Proposition \ref{propScottCoarser}). Finally, we prove our first main result.

In Section \ref{SecTreeSME}, we present an overview of the construction of the subtree $\mathcal{T}_{\rm sme}$. In Section \ref{SecUpperBoundsSuprema}, we construct upper bounds for increasing families of valuations $\mathfrak{v}=\{\nu_i\}_{i\in I}$ on the tree $\mathcal{T}$ and define the supremum $\sup_{i\in I}\nu_i$ of $\mathfrak{v}$.  Then we prove our second main result.  We also show that $\sup_{i\in I}\nu_i$  can be seen as a limit of the family $\mathfrak{v}$ in the topologies presented in Section \ref{SecLambdaTree} (see Propositions \ref{PropSupLimScott} and \ref{propLimToSup}).

We begin Section \ref{SecClosenessCriterion} by proving our third main result. Then we study the induced order topology on the space $\mathcal{V}_{R}(\Lambda)$ and conclude that the tree  $\mathcal{T}(v,\Lambda)$ is closed in the product  $(\Lambda_\infty)^{K[x]}$ when considering the order topology on $\Lambda_\infty$ (see Corollary \ref{CorTreeisClosed}).


\section{Preliminaries}\label{SecPreliminaries}

Consider a commutative ring $R$ with unity and a totally ordered abelian group $\Gamma$. We denote by $\Gamma_{\infty}$ the set $\Gamma\cup \{\infty\}$ where $\infty$ is a symbol not belonging to $\Gamma$. We extend the addition and the order to $\Gamma_{\infty}$ by defining:
\begin{itemize}
\item $\gamma + \infty = \infty + \gamma = \infty$ for all $\gamma \in \Gamma_{\infty}$.

\item $\gamma < \infty$ for all $\gamma \in \Gamma$.

\end{itemize}

\begin{Def}
We say that a map $v: R\longrightarrow \Gamma_{\infty}$ is a \textbf{valuation} if it satisfies the following conditions:
\begin{description}
	\item [(V1)] $v(ab)=v(a)+v(b)$ for all $a, b \in R$.
	
	\item [(V2)] $v(a+b)\geq \min \{ v(a), v(b) \}$ for all $a, b \in R$.
	
	\item [(V3)] $v(1)=0$ and $v(0)=\infty$.
	
\end{description}
\end{Def}

The subgroup of $\Gamma$ generated by the set $\{v(a)\mid a \in R \text{ and }v(a)\neq \infty\}$ is called the \textbf{value group} of $v$ and is denoted by $\Gamma_v$. In the case where $R$ is a field, we have
\[
\Gamma_v=\{v(a)\mid a \in R \text{ and }v(a)\neq \infty\}.
\]
The set $\SU(v):=v^{-1}(\infty)$ is a prime ideal of $R$, called the \textbf{support} of $v$. If $\SU(v)=(0)$, then we say that $v$ is a \textbf{Krull} valuation (or that it has \textbf{trivial support}). Otherwise, the valuation is said to be \textbf{non-Krull} (or that it has \textbf{non-trivial support}).

 If $v$ is a Krull valuation, then $R$ is a domain and we can extend $v$ to $ K={\rm Quot}(R)$ by additivity (axiom \textbf{(V1)}). In this case, 
define the \textbf{valuation ring} of $v$ as 
$$\VR_v:=\{ a\in K\mid v(a)\geq 0 \}.$$ 
The ring $\VR_v$ is a local ring with unique maximal ideal $\MI_v=\{a\in K\mid v(a)>0 \}. $ We define the \textbf{residue field} of $v$ as the field $\VR_v/\MI_v$ and denote it by $ k_v$. The image of $a\in \VR_v$ in $ k_v$ is denoted by $av$.


\begin{Def}
We say that the valuations $v$ and $u$ on $R$ are \textbf{equivalent} (denoted $v\sim u$) if the following equivalent conditions are satisfied:
\begin{enumerate}
\item[\textbf{i)}] There exists an order-preserving isomorphism $\Phi: \Gamma_{u}\longrightarrow \Gamma_{v}$ such that $v=\Phi\circ u$.
\item[\textbf{ii)}] For all $a, b \in R$, $v(a)>v(b)$ if and only if $u(a)>u(b)$.
\end{enumerate}
We denote the equivalence class of $v$ by $[v]$.
\end{Def}


Let $(K,v)$ be a \textbf{valued field}, i.e., $K$ is a field and $v$ is a fixed valuation on $K$.  Consider 
$$\mathcal{V}=\mathcal{V}(v)=\{\nu:K[x]\to (\Gamma_\nu)_\infty\mid \nu \text{ is a valuation and } \nu|_K\sim v\},$$
  the set of all valuations on $K[x]$ whose restriction to $K$ is equivalent to $v$. We have an embedding $\Gamma_v\hookrightarrow \Gamma_\nu$ for every $\nu \in \mathcal{V}$. For simplicity of notation, we will consider this embedding to be an inclusion. Hence, $\mathcal{V}$ is the set of all valuations that extend $v$ to  $K[x]$.



Given $\mu \in \mathcal{V}$ with $\SU(\mu)=(0)$,
since $\mathcal{O}_v\subseteq  \mathcal{O}_{\mu}$ and $\mathfrak{m}_v= \mathfrak{m}_{\mu}\cap \mathcal O_v$, it follows that $k_{\mu}$ is a field extension of $k_v$. We can classify $\mu$ as follows:
\begin{itemize}
\item \textbf{Value-transcendental}, if $\Gamma_{\mu}/\Gamma_v$ is not a torsion group.
\item \textbf{Residue-transcendental}, if the extension $k_{\mu}/k_v$ is transcendental.
\item \textbf{Valuation-algebraic}, if $\Gamma_{\mu}/\Gamma_v$ is a torsion group and $k_{\mu}/k_v$ is algebraic.
\end{itemize}

We say that $\mu \in \mathcal{V}$ is \textbf{valuation-transcendental} if it is either value-transcendental or residue-transcendental. 

\begin{Obs}\label{desigualdadeAB}
    By Abhyankar's inequality (see [\cite{Zar_1}, p. 330]), a valuation cannot be simultaneously value-transcendental and residue-transcendental. Therefore, $\Gamma_{\mu}/\Gamma_v$ is a torsion group if $\mu$ is residue-transcendental.
\end{Obs}

Take the \textbf{divisible hull} $\Gamma_{\mathbb{Q}} := \Gamma_v \otimes_{\mathbb{Z}} \mathbb{Q}$ of $\Gamma_v$.


\begin{Obs}\label{obs2}$\,$

\begin{itemize}
    \item For $\mu \in \mathcal{V}$, if $\mu$ is valuation-algebraic, then by definition $\Gamma_{\mu} \subseteq \Gamma_{\mathbb{Q}}$. By Remark~\ref{desigualdadeAB}, the same inclusion $\Gamma_{\mu} \subseteq \Gamma_{\mathbb{Q}}$ holds when $\mu$ is residue-transcendental.

    \item By direct calculations, one can prove that if $\Gamma_{\mu}/\Gamma_v$ is not torsion, then $\SU(\mu)=(0)$. That is, value-transcendental valuations are the only ones whose image is not contained in the $\Gamma_{\mathbb{Q}}$.
\end{itemize}
	
\end{Obs}

Next lemma follows from direct calculations. 

\begin{Lema}\label{equi}
	Let $\mu, \nu$ be equivalent valuations. Then:
	\begin{enumerate}
		\item[(a)] If $\mu$ is not a Krull valuation, then $\nu$ is not a Krull valuation.
		
		\item[(b)] If $\mu$ is value-transcendental, then $\nu$ is value-transcendental.
		
		\item[(c)] If $\mu$ is residue-transcendental, then $\nu$ is residue-transcendental.
		
		\item[(d)] If $\mu$ is valuation-algebraic, then $\nu$ is valuation-algebraic.
	\end{enumerate}
\end{Lema}

\begin{Obs}\label{equi_fecho}
Given $\mu, \nu \in \mathcal{V}$ such that $\mu \sim \nu$ and $\Gamma_{\mu}, \Gamma_{\nu} \subset \Gamma_{\mathbb{Q}}$, it follows from a direct calculation 
that $\mu = \nu$.
\end{Obs}

	
	
	

\section{The set $\mathbb{V}$}\label{SecSpaceOfValuation}

One of the objects of our study is the set
\[
\mathbb{V}=\mathbb{V}(v) := \{[\mu] \mid \mu \in \mathcal{V}(v)\}
\]
consisting of all equivalence classes of valuations that extend $v$.
We will equip $\mathbb{V}$ with a topology, as follows.



\begin{Def}
	The \textbf{valuation spectrum topology} on $\mathbb{V}$ is the topology having as a subbasis the sets of the form
	$$U(f,g):=\{[\nu]\in \mathbb{V}\mid \nu(f)\geq \nu(g)\neq \infty\} $$
	where $f$ and $g$ run over $K[x]$.
\end{Def}

Let $\mathbb{V}_{\mathrm{nt}}$, $\mathbb{V}_{\mathrm{vt}}$, $\mathbb{V}_{\mathrm{rt}}$ and $\mathbb{V}_{\mathrm{al}}$ be the subsets of $\mathbb{V}$ consisting of classes of nontrivial support, value-transcendental, residue-transcendental and valuation-algebraic valuations, respectively. A consequence of Lemma \ref{equi} is that
\[
\mathbb{V} = \mathbb{V}_{\mathrm{nt}} \sqcup \mathbb{V}_{\mathrm{vt}} \sqcup \mathbb{V}_{\mathrm{rt}} \sqcup \mathbb{V}_{\mathrm{al}}.
\]
We want to study the topological behavior of the sets $\mathbb{V}_{\mathrm{vt}}$, $\mathbb{V}_{\mathrm{rt}}$  inside $\mathbb{V}$. For this, we will need  a lemma.

Given a valuation $\nu$ on $K[x]$ and $q\in K[x]$ monic non-constant, consider the map 
$$\nu_q(f):=\underset{0\leq j\leq r}{\min} \{ \nu(f_jq^j)  \},$$
for $f = f_0+f_1q+\ldots+f_rq^r \in K[x]$ written in its $q$-expansion (i.e. $\deg(f_j)<\deg(q)$ for every $j$, $0\leq j \leq r$). This map is called the \textbf{truncation} of $\nu$ at $q$.

\begin{Lema}\label{lemqnSobrecTransc}
	Suppose that there exist $q\in K[x]$, $c\in K^{\times}$ and $n\in \N$ such that $\nu=\nu_q$ and $n\nu(q)=v(c)$. Then $(q^n/c)\nu\in k_\nu$ is transcendental over $k_v$. In particular, $\nu$ is residue-transcendental.
\end{Lema}

\begin{proof}
	Suppose there exist	 $b_0, \ldots , b_r\in \mathcal{O}_{v}$ such that
	$$ \sum_{j=0}^r (b_j\nu)\left(\frac{q^n}{c}\nu\right)^j = 0 \text{ in } k_\nu.$$
	Hence,
	$$\nu \left(\sum_{j=0}^r b_j\left(\frac{q^n}{c}\right)^j\right)>0. $$
	Since $\nu=\nu_q$, for every $j$, $0\leq j\leq r$,  we have
	$$\nu\left(b_j\left(\frac{q^n}{c}\right)^j\right)\geq \underset{0\leq i\leq r}{\min} \left\lbrace\nu\left(b_i\left(\frac{q^n}{c}\right)^i\right)  \right\rbrace = \nu \left(\sum_{i=0}^r b_i\left(\frac{q^n}{c}\right)^i\right)>0. $$
	Hence,
	$ (b_j\nu)((q^n/c)\nu)^j = (b_j(q^n/c)^j)\nu = 0$. Since $(q^n/c)\nu \neq 0$, it follows that  $b_j\nu =0$ for every $j$. We conclude that  $(q^n/c)\nu$ is transcendental over $k_v$. In particular, $\nu$ is residue-transcendental.
	
\end{proof}

\begin{Prop}\label{PropVTRTDenseValSpecTop}
	The sets $\mathbb V_{\rm vt}$ and $\mathbb V_{\rm rt}$ are dense in $\mathbb{V}$ with the valuation spectrum topology. 
\end{Prop}

\begin{proof}Take a non-empty basic open set $U(f,g)$. We need to show that there exists at least one residue-transcendental valuation class and one value-transcendental valuation class in $U(f,g)$. 
	
	If $f/g\in K$, then $f=ag$ for some $a\in K$. Take any $[\nu]\in U(f,g)$.   Hence, $\nu(a)+\nu(g)=\nu(f)\geq \nu(g)$, which implies $\nu(a)\geq 0$. Since $\nu|_K=v$, we have $v(a)\geq 0$. Take any $[\nu']\in\mathbb{V}$. Since $\nu'|_K=v$, we also have $\nu'(a)\geq 0$. Hence, $\nu'(f)\geq \nu'(g)$ and $[\nu']\in U(f,g)$. Therefore, $U(f,g)=\mathbb{V}$ and we are done.

	Suppose $h=f/g\not \in K$.
	The extension $K(h)\mid K$ is transcendental and the extension $K(x)\mid K(h)$ is algebraic.  We will construct a valuation on $K(h)$ with the desired property and then take an extension to $K(x)$. 
	
	\begin{itemize}
		\item 	\noindent A value-transcendental valuation in $U(f,g)$: Consider the map 
		$$\nu'(a_rh^r+\ldots+ a_0):=\min_{0\leq j \leq r}\{ v(a_j) +j(1,0) \}.  $$
		This is a valuation on $K(h)$ with values in $\Z\oplus \Gamma$ such that $\nu'(h)=(1,0)$ is not a torsion element over $\Gamma\cong (0)\oplus \Gamma$. Take any extension $\nu$ of $\nu'$ to $K(x)$. The valuation $\nu$ is  value-transcendental. 
		
		\item 	\noindent  A residue-transcendental valuation in $U(f,g)$: Consider the map
		$$\nu'(a_rh^r+\ldots+ a_0):= \min_{0\leq j \leq r} \{v(a_j)\}. $$
		This is a valuation on $K(h)$ (it is the \textit{Gauss valuation} with $h$ in the place of $x$). By the definition of truncation, we have $(\nu')_h=\nu'$ and $\nu'(h)=\nu'(1)=0$. By the same reasoning used in the proof of Lemma \ref{lemqnSobrecTransc}, we see that $h\nu'\in k_{\nu'}$ is transcendental over $k_v$. We extend $\nu'$ to $\nu$ in $K(x)$ and we still have $k_\nu/k_v$ transcendental. 
		\smallskip

	\end{itemize}
	
	In both cases, $[\nu]\in \mathbb{V}$ and $\nu(h)\geq 0$. Hence, $[\nu]\in U(f,g)$ and  the result follows.
	
\end{proof}

Now consider the subset of Krull valuations
\[
\mathbb{V}_{\mathrm{Krull}}:=\mathbb{V}_{\mathrm{vt}} \sqcup \mathbb{V}_{\mathrm{rt}} \sqcup \mathbb{V}_{\mathrm{al}}\subseteq\mathbb{V},
\]
on which we will define the Zariski topology.

\begin{Def}
	The \textbf{Zariski topology} on $\mathbb{V}_{\mathrm{Krull}}$ is the restriction of the valuation spectrum topology from $\mathbb{V}$.
\end{Def}

\begin{Obs}
One can see that 
	$$U(f):=\{[\nu]\in \mathbb{V}_{\mathrm{Krull}}\mid \nu(f)\geq 0\} $$
	where $f$ run over $K(x)$ form a subbasis for the Zariski topology.
\end{Obs}

We have the following corollary as a consequence of Proposition~\ref{PropVTRTDenseValSpecTop}.




\begin{Cor}
	The sets $\mathbb V_{\rm vt}$ and $\mathbb V_{\rm rt}$ are dense in $\mathbb{V}_{\mathrm{Krull}}$ with the Zariski topology. 
\end{Cor}

 
\section{Topologies on $\Lambda$-trees}\label{SecLambdaTree}

\subsection{$\Lambda$-trees}

Fix an ordered abelian group $\Lambda$ containing $\Gamma_v$. Consider the following subset of $\mathcal{V}= \mathcal{V}(v)$:
\[
\mathcal{T}=\mathcal{T}(v, \Lambda) := \{\mu \in \mathcal{V} \mid \Gamma_{\mu} \subseteq \Lambda\}.
\]
In some cases, we will consider $\Lambda$ to be a monoid; unless otherwise stated, assume it is a group. The set $\mathcal{T}$ carries a partial order: for all $\mu, \nu \in \mathcal{T}$, we define
\[
\mu \leq \nu \quad \Longleftrightarrow \quad \mu(f) \leq \nu(f) \text{ for all } f \in K[x].
\]
Using the terminology of \cite{Nart_1}, we say that $\mathcal{T}$ is a ($\Lambda$-)\textbf{tree} since the intervals
\[
(-\infty, \mu] := \{\nu \in \mathcal{T} \mid \nu \le \mu\}
\]
are totally ordered for every $\mu \in \mathcal{T}$  \cite[Theorem 2.4]{Nart_2}. We will also refer to $\mathcal{T}$ as a \textbf{valuative tree}.

\subsection{The weak tree topology}

Let $\Lambda$ be an ordered abelian group containing the divisible hull $\Gamma_{\mathbb{Q}}$ of $\Gamma_v$. Given $\mu, \nu \in \mathcal{T}(v, \Lambda)$, denote by $$\mu\wedge \nu:=\inf\{\mu, \nu\}=\max\{(-\infty, \mu]\cap (-\infty, \nu]\},$$ if it exists.

\begin{Prop}\cite[Proposition 5.2]{Nart_1}\label{LemExisteJoint}
For all $\mu, \nu \in \mathcal{T}(v, \Lambda)$, there exists $\mu\wedge\nu$.
\end{Prop}

With the existence of the infimum, we can define the \textbf{interval} between two valuations in $\mathcal{T}$ as:
\[
[\mu, \nu] := \{\eta \in \mathcal{T} \mid \mu\wedge\nu \leq \eta \leq \mu \ \text{or} \ \mu\wedge\nu \leq \eta \leq \nu\}.
\]
That is, $[\mu, \nu]$ is the union of the segments $[\mu\wedge\nu, \mu]$ and $[\mu\wedge\nu, \nu]$
. We define $(\mu, \nu)$ similarly.

\begin{figure}[htbp]
    \centering
    
    \begin{subfigure}[b]{0.48\textwidth}
        \centering
        \definecolor{qqqqff}{rgb}{0,0,1}
        \definecolor{ududff}{rgb}{0.30196078431372547,0.25}
        \begin{tikzpicture}[line cap=round,line join=round,>=triangle 45,x=0.55cm,y=0.55cm]
        \clip(-0.8568828236857945,-0.15) rectangle (8.871502224064375,9.2);
        \draw [line width=0.5pt,color=qqqqff] (4,3)-- (4,6); 
        \draw [line width=0.5pt] (4,6)-- (4,8);
        \draw [line width=0.5pt] (4,3)-- (4,1);
        \draw [line width=0.5pt] (2.68,5.6)-- (4,3.82);
        \draw [line width=0.5pt] (4,6.6)-- (5.18,7.86);
        \begin{scriptsize}
        \draw [fill=qqqqff] (4,3) circle (2pt);
        \draw [fill=qqqqff] (4,6) circle (2pt);
        \draw[color=black] (4.444918716178052,6.141848546137132) node {$\nu$};
        \draw[color=black] (5.3,3.1290787822145134) node {$\mu=\mu\wedge\nu$};
        \end{scriptsize}
        \end{tikzpicture}
        \caption{Comparable valuations}
        \label{fig:comp}
    \end{subfigure}
    \hfill
    \begin{subfigure}[b]{0.48\textwidth}
        \centering
        \definecolor{qqqqff}{rgb}{0,0,1}
        \definecolor{ududff}{rgb}{0.30196078431372547,0.25}
        \begin{tikzpicture}[line cap=round,line join=round,>=triangle 45,x=0.55cm,y=0.55cm]
        \clip(-0.9578695196832009,-1.0) rectangle (8.770515528066966,8.0);
        \draw [line width=0.5pt] (1.5836289962515313,7.008651020114868)-- (2.37469,5.98195);
        \draw [line width=0.5pt,color=qqqqff] (2.37469,5.98195)-- (4,4); 
        \draw [line width=0.5pt] (6.380497056128343,7.008651020114868)-- (5.58943,5.98195);
        \draw [line width=0.5pt,color=qqqqff] (5.58943,5.98195)-- (4,4); 
        \draw [line width=0.5pt] (4,4)-- (4,1);
        \draw [line width=0.5pt] (4,1)-- (4,0);
        \begin{scriptsize}
        \draw [fill=qqqqff] (4,4) circle (2pt);
        \draw [fill=qqqqff] (5.58943,5.98195) circle (2pt);
        \draw [fill=qqqqff] (2.37469,5.98195) circle (2pt);
        \draw[color=black] (1.8,6.00719961814059) node {$\mu$};
        \draw[color=black] (6.1,6.00719961814059) node {$\nu$};
        \draw[color=black] (4.9,4.0) node {$\mu\wedge\nu$};
        \end{scriptsize}
        \end{tikzpicture}
        \caption{Non-comparable valuations}
        \label{fig:noncomp}
    \end{subfigure}
    
    \caption{Interval representation in the valuative tree}
    \label{fig:valuations}
\end{figure}
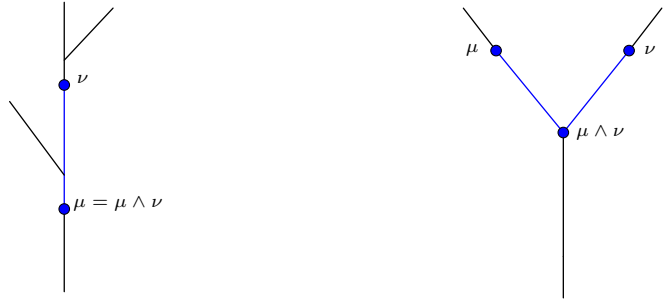

The next lemma gives a property of intervals that we will need in order to prove the main result of this section.

\begin{Lema}\label{lema1}
Let $\mu, \nu, \eta \in \mathcal{T}(v, \Lambda)$. Then we have
$[\mu, \eta] \subseteq [\mu, \nu] \cup [\nu, \eta]$.
\end{Lema}

\begin{proof}
It suffices to show that both intervals $[\mu \wedge \eta, \mu]$ and $[\mu \wedge \eta, \eta]$ are contained in $[\mu, \nu] \cup [\nu, \eta]$.
Consider the interval $(-\infty, \nu]$, which is totally ordered. Since $\nu \wedge \eta \leq \nu$ and $\mu \wedge \nu \leq \nu$, these two valuations are comparable.

If $\mu \wedge \nu \leq \nu \wedge \eta$, then $\mu \wedge \nu \leq \eta$. Hence, $\mu \wedge \nu \leq \mu \wedge \eta$, and so
\[
[\mu \wedge \eta, \mu] \subseteq [\mu \wedge \nu, \mu] \subseteq [\mu, \nu].
\]

If $\nu \wedge \eta \leq \mu \wedge \nu$, then $\nu \wedge \eta \leq \mu$ and consequently $\nu \wedge \eta \leq \mu \wedge \eta$. On the other hand, since both $\mu \wedge \nu$ and $\mu \wedge \eta$ are smaller or equal to $\mu$, and the interval $(-\infty, \mu]$ is totally ordered, we must have
\[
\mu \wedge \nu \leq \mu \wedge \eta \quad \text{or} \quad \mu \wedge \eta \leq \mu \wedge \nu.
\]

In the first case, we proceed as before and conclude that $[\mu \wedge \eta, \mu] \subseteq [\mu, \nu]$. 
In the second case, we have $\mu \wedge \eta \leq \nu$, and thus $\mu \wedge \eta \leq \nu \wedge \eta$. Therefore, $\mu \wedge \eta = \nu \wedge \eta$, and we obtain
\[
[\mu \wedge \eta, \mu \wedge \nu] = [\nu \wedge \eta, \mu \wedge \nu] \subseteq [\nu \wedge \eta, \nu] \subseteq [\nu, \eta].
\]
Hence,
\[
[\mu \wedge \eta, \mu] = [\mu \wedge \eta, \mu \wedge \nu] \cup [\mu \wedge \nu, \mu] \subseteq [\nu, \eta] \cup [\mu, \nu].
\]
The proof that $[\mu \wedge \eta, \eta] \subseteq [\mu, \nu] \cup [\nu, \eta]$ is analogous.
\end{proof}

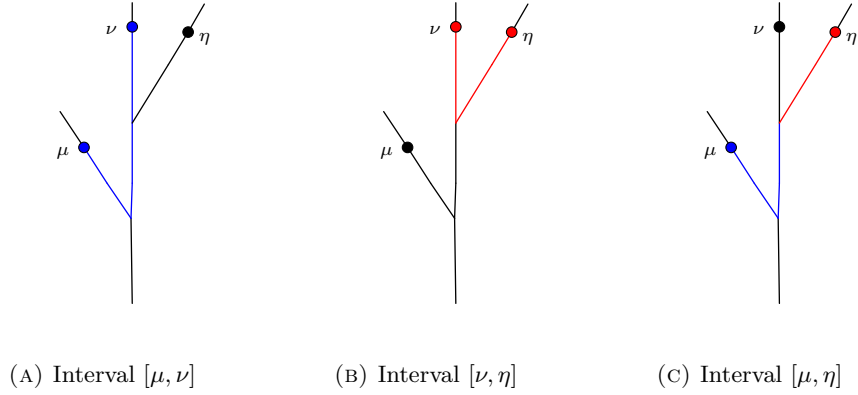
\begin{figure}[h]
    \centering
    \begin{subfigure}{0.32\textwidth}
        \centering
        \definecolor{qqqqff}{rgb}{0,0,1}
        \definecolor{ududff}{rgb}{0.30196078431372547,0.30196078431372547,1}
        \begin{tikzpicture}[line cap=round,line join=round,>=triangle 45,x=0.8cm,y=0.8cm]
        \clip(1.9739234624800106,2.348750338254921) rectangle (8.102295851164635,8.583149930896012);
        \draw [line width=0.5pt] (5,3)-- (4.979665721613332,4.415452098029648);
        \draw [line width=0.5pt,color=qqqqff] (4.979665721613332,4.415452098029648)-- (4.582341378633048,4.998194467734064);
        \draw [line width=0.5pt,color=qqqqff] (4.582341378633048,4.998194467734064)-- (4.198261180418773,5.59418098220449);
        \draw [line width=0.5pt] (4.198261180418773,5.59418098220449)-- (3.800936837438489,6.190167496674916);
        \draw [line width=0.5pt,color=qqqqff] (4.979665721613332,4.415452098029648)-- (5,5);
        \draw [line width=0.5pt,color=qqqqff] (5,5)-- (5,6);
        \draw [line width=0.5pt,color=qqqqff] (5,6)-- (5,7.6);
        \draw [line width=0.5pt] (5,7.6)-- (5,8);
        \draw [line width=0.5pt] (5,6)-- (5.600053837791674,6.9715364307574985);
        \draw [line width=0.5pt] (5.600053837791674,6.9715364307574985)-- (5.928738170056767,7.512275170935552);
        \draw [line width=0.5pt] (5.928738170056767,7.512275170935552)-- (6.198127040086204,7.978127040086194);
        \begin{scriptsize}
        \draw [fill=qqqqff] (4.198261180418773,5.59418098220449) circle (2pt);
        \draw [fill=qqqqff] (5,7.6) circle (2pt);
        \draw[color=black] (3.85,5.5) node {$\mu$};
        \draw[color=black] (4.656411722578991,7.55) node {$\nu$};
        \draw [fill=black] (5.928738170056767,7.512275170935552) circle (2pt);
        \draw[color=black] (6.2,7.4) node {$\eta$};
        \end{scriptsize}
        \end{tikzpicture}
        \caption{Interval $[\mu, \nu]$}
    \end{subfigure}
    \hfill
    \begin{subfigure}{0.32\textwidth}
        \centering
        \definecolor{qqqqff}{rgb}{0,0,1}
        \definecolor{ududff}{rgb}{1,0,0}
        \begin{tikzpicture}[line cap=round,line join=round,>=triangle 45,x=0.8cm,y=0.8cm]
        \clip(1.9739234624800106,2.348750338254921) rectangle (8.102295851164635,8.583149930896012);
        \draw [line width=0.5pt] (5,3)-- (4.979665721613332,4.415452098029648);
        \draw [line width=0.5pt] (4.979665721613332,4.415452098029648)-- (4.582341378633048,4.998194467734064);
        \draw [line width=0.5pt] (4.582341378633048,4.998194467734064)-- (4.198261180418773,5.59418098220449);
        \draw [line width=0.5pt] (4.198261180418773,5.59418098220449)-- (3.800936837438489,6.190167496674916);
        \draw [line width=0.5pt] (4.979665721613332,4.415452098029648)-- (5,5);
        \draw [line width=0.5pt] (5,5)-- (5,6);
        \draw [line width=0.5pt,color=ududff] (5,6)-- (5,7.6);
        \draw [line width=0.5pt] (5,7.6)-- (5,8);
        \draw [line width=0.5pt,color=ududff] (5,6)-- (5.600053837791674,6.9715364307574985);
        \draw [line width=0.5pt,color=ududff] (5.600053837791674,6.9715364307574985)-- (5.928738170056767,7.512275170935552);
        \draw [line width=0.5pt] (5.928738170056767,7.512275170935552)-- (6.198127040086204,7.978127040086194);
        \begin{scriptsize}
        \draw [fill=black] (4.198261180418773,5.59418098220449) circle (2pt);
        \draw [fill=ududff] (5,7.6) circle (2pt);
        \draw[color=black] (3.85,5.5) node {$\mu$};
        \draw[color=black] (4.656411722578991,7.55) node {$\nu$};
        \draw [fill=ududff] (5.928738170056767,7.512275170935552) circle (2pt);
        \draw[color=black] (6.2,7.4) node {$\eta$};
        \end{scriptsize}
        \end{tikzpicture}
        \caption{Interval $[\nu, \eta]$}
    \end{subfigure}
    \hfill
    \begin{subfigure}{0.32\textwidth}
        \centering
        \definecolor{qqqqff}{rgb}{0,0,1}
        \definecolor{ududff}{rgb}{1,0,0}
        \begin{tikzpicture}[line cap=round,line join=round,>=triangle 45,x=0.8cm,y=0.8cm]
        \clip(1.9739234624800106,2.348750338254921) rectangle (8.102295851164635,8.583149930896012);
        \draw [line width=0.5pt] (5,3)-- (4.979665721613332,4.415452098029648);
        \draw [line width=0.5pt,color=qqqqff] (4.979665721613332,4.415452098029648)-- (4.582341378633048,4.998194467734064);
        \draw [line width=0.5pt,color=qqqqff] (4.582341378633048,4.998194467734064)-- (4.198261180418773,5.59418098220449);
        \draw [line width=0.5pt] (4.198261180418773,5.59418098220449)-- (3.800936837438489,6.190167496674916);
        \draw [line width=0.5pt,color=qqqqff] (4.979665721613332,4.415452098029648)-- (5,5);
        \draw [line width=0.5pt,color=qqqqff] (5,5)-- (5,6);
        \draw [line width=0.5pt] (5,6)-- (5,7.6);
        \draw [line width=0.5pt] (5,7.6)-- (5,8);
        \draw [line width=0.5pt,color=ududff] (5,6)-- (5.600053837791674,6.9715364307574985);
        \draw [line width=0.5pt,color=ududff] (5.600053837791674,6.9715364307574985)-- (5.928738170056767,7.512275170935552);
        \draw [line width=0.5pt] (5.928738170056767,7.512275170935552)-- (6.198127040086204,7.978127040086194);
        \begin{scriptsize}
        \draw [fill=qqqqff] (4.198261180418773,5.59418098220449) circle (2pt);
        \draw [fill=black] (5,7.6) circle (2pt);
        \draw[color=black] (3.85,5.5) node {$\mu$};
        \draw[color=black] (4.656411722578991,7.55) node {$\nu$};
        \draw [fill=ududff] (5.928738170056767,7.512275170935552) circle (2pt);
        \draw[color=black] (6.2,7.4) node {$\eta$};
        \end{scriptsize}
        \end{tikzpicture}
        \caption{Interval $[\mu, \eta]$}
    \end{subfigure}
    \caption{Representation of Lemma~\ref{lema1} in the valuative tree.}
\end{figure}


Given $\mu \in \mathcal{T}$, consider the following equivalence relation on
$\mathcal{T} \setminus \{\mu\}$:
\[
\nu \sim_{\mu} \eta \Longleftrightarrow \mu \notin [\nu, \eta].
\]

For each $\nu \in \mathcal{T} \setminus \{\mu\}$, define the equivalence class
\[
[\nu]_{\mu} := \{\eta \in \mathcal{T} \mid \nu \sim_{\mu} \eta\}.
\]
Denote by
\[
\mathcal{T}_{\mu} := \{ [\nu]_{\mu} \mid \nu \in \mathcal{T} \setminus \{\mu\} \}
\]
the set of all such equivalence classes. This equivalence relation, inspired by the
work of Favre and Jonsson on the valuative tree of centered valuations~\cite{Fav_1},
is here extended to all valuations in $\mathcal{T}$.



\begin{Def}
The \textbf{weak tree topology} $\mathfrak{W}$ on $\mathcal{T}$ is the topology generated by all sets of the form $[\nu]_{\mu}$, where $\mu$ runs through $\mathcal T$.
\end{Def}

Fixing $\mu \in \mathcal{T}$, we define a relation $\leq_{\mu}$ on $\mathcal{T}$ as follows: given two valuations $\nu, \eta \in \mathcal{T}$, we say that $$\nu \leq_{\mu} \eta \Longleftrightarrow \nu \in [\mu, \eta].$$ Note that $\mu$ is the smallest element of $(\mathcal{T}, \leq_{\mu})$, since $\mu \in [\mu, \nu]$ for all $\nu \in \mathcal{T}$; thus, we say that $\mu$ is the \textbf{root} of $(\mathcal{T}, \leq_{\mu})$. Moreover, the segments under this new order are exactly the same as those defined by the order $\leq$. Consequently, the weak tree topology on $\mathcal T$ defined by these two orders is the same.

\subsubsection{Key polynomials and tangent directions}

Given $\mu \in \mathcal{T}(v, \Lambda)$, the \textbf{graded algebra} of $\mu$ is the integral domain $\mathcal{G}_{\mu}=\bigoplus_{\alpha \in \Gamma_{\mu}}\mathcal{P}_{\alpha}/\mathcal{P}^+_{\alpha}$, where
\[\mathcal{P}_{\alpha}=\{f \in K[x]\mid\mu(f)\geq\alpha\}\text{ and }\mathcal{P}^+_{\alpha}=\{f \in K[x]\mid\mu(f)>\alpha\}.\]

Consider the \textbf{initial term} map ${\rm in }_{\mu}: K[x]\longrightarrow\mathcal{G}_{\mu}$, given by ${\rm in}_{\mu}(f)=0$ for every $f \in \SU(\mu)$ and
\[{\rm in}_{\mu}f=f+\mathcal{P}^+_{\mu(f)}\in \mathcal{P}_{\mu(f)}/\mathcal{P}^+_{\mu(f)}\text{, if }f \in K[x]\setminus \SU(\mu).\]

\begin{Def}
A monic $\phi\in K[x]$ is a \textbf{(Mac Lane-Vaquié) key polynomial} for $\mu$ if $({\rm in}_{\mu}\phi)\mathcal{G}_{\mu}$ is a homogeneous prime ideal containing no initial term ${\rm in}_{\mu}f$ with $\deg f<\deg \phi$.
\end{Def}

We denote by KP($\mu$) the set of all key polynomials of $\mu$. These polynomials are always irreducible in $K[x]$. We present a criterion for the existence of key polynomials that was proved in \cite[Theorem 2.3]{Nart_2} and \cite[Theorem 4.4]{NartKeyPolyValuedFields}.

\begin{Teo}\label{MLteo2.3}
Let $\mu \in \mathcal{T}(v, \Lambda)$. Then, the following are equivalent:
\begin{enumerate}
    \item $\operatorname{KP}(\mu) = \emptyset$;
    \item $\mu$ is either valuation-algebraic or has nontrivial support;
    \item $\mu$ is maximal in $\mathcal{T}(v, \Lambda)$.
\end{enumerate}
\end{Teo}


Given a non-maximal valuation $\mu \in \mathcal{T}(v, \Lambda)$, for any $\phi \in \operatorname{KP}(\mu)$ and $\gamma \in \Lambda_{\infty}$ such that $\mu(\phi) < \gamma$, we consider the \textbf{augmented valuation} $\nu = [\mu; \phi, \gamma]$ given by
\[
\nu(f) = \min_{0 \leq j \leq r}\{ \mu(f_j) + j\gamma \},
\]
where $f = \sum_{j=0}^r f_j \phi^{j}$ with $\deg(f_j) < \deg(\phi)$. In \cite[Theorem 1.1]{Mvaquie}, it is proved that this map is indeed a valuation. Moreover, one can easily verify that
\begin{equation}
\nu \text{ is } \left\{
\begin{aligned}
&\text{a non-trivial support valuation,} && \text{if } \gamma = \infty; \\
&\text{residue-transcendental,} && \text{if } \gamma \in \Gamma_{\mathbb{Q}}; \\
&\text{value-transcendental,} && \text{if } \gamma \in \Lambda \setminus \Gamma_{\mathbb{Q}}.
\end{aligned}
\right.
\label{Vtipos}
\end{equation}

\begin{Def}
Take two valuations $\mu, \nu \in \mathcal{T}$ that are not maximal, with $\mu < \nu$. 
The set $\mathbf{t}(\mu, \nu)$ of monic polynomials $\phi \in K[x]$ of minimal degree satisfying $\mu(\phi) < \nu(\phi)$ is called the \textbf{tangent direction} of $\mu$ determined by $\nu$.
\end{Def}

Now, fix a non-maximal valuation $\mu \in \mathcal{T}$. Consider the following equivalence relation on $\mathcal{T}$:
\[
\nu \sim_{\mathrm{tan}} \eta \;\Longleftrightarrow\; \nu > \mu\text{, }\eta > \mu\text{ and }(\mu, \nu] \cap (\mu, \eta] \neq \emptyset,
\]
where $\nu, \eta \in \mathcal{T}$. Define the set
\[
\mathrm{TD}(\mu) = \{\nu \in \mathcal{T} \mid \mu < \nu\}/\sim_{\mathrm{tan}}
\]
and denote the equivalence class of $\nu$ in $\mathrm{TD}(\mu)$ by $[\nu]_{\mathrm{tan}}$. The following result can be found in \cite[Proposition 2.4]{Nart_1}.

\begin{Prop}\label{prop:tangent_relation}
Let $\mu, \nu, \eta \in \mathcal{T}$ be such that $\mu < \nu$ and $\mu < \eta$. Then,
\[
\mathbf{t}(\mu, \nu) = \mathbf{t}(\mu, \eta) \quad \text{if and only if} \quad \nu \sim_{\mathrm{tan}} \eta.
\]
\end{Prop}

The next proposition characterizes the elements of $\mathcal{T}_{\mu}$ in terms of tangent directions. Before we present it, let us introduce a notation inspired by the following lemma. 

\begin{Lema}\label{ConjB}
  For \(\mu, \nu, \eta \in \mathcal{T}(v, \Lambda)\), if \(\mu \not\leq \nu\), then \(\nu \sim_{\mu} \eta\) if and only if \(\mu \not\leq \eta\).
\end{Lema}

\begin{proof}
  Suppose, aiming for contradiction, that \(\mu \leq \eta\).  
  Since \(\nu \wedge \eta \leq \eta\), \(\mu \leq \eta\) and the set \((- \infty, \eta]\) is totally ordered, we have
  \[
  \nu \wedge \eta \leq \mu \quad \text{or} \quad \mu < \nu \wedge \eta.
  \]
  If \(\mu < \nu \wedge \eta\), then \(\mu < \nu\), which is a contradiction.  
  Hence, \(\nu \wedge \eta \leq \mu \leq \eta\) and thus \(\mu \in [\nu, \eta]\), contradicting the fact that \(\nu \sim_{\mu} \eta\).  
  Therefore, \(\mu \not\leq \eta\).

  For the converse, since \(\mu \not\leq \nu\) and \(\mu \not\leq \eta\), it follows that \(\mu \notin [\nu, \eta]\). Hence, \(\nu \sim_{\mu} \eta\), as we wanted to prove.
\end{proof}

In this case, we define
\[
\mathcal{B}_{\nleq}(\mu):=\{\rho \in \mathcal{T} \mid \mu \nleq \rho\}=[\eta]_{\mu},
\]
where $\eta \in \mathcal{T}$ is any element such that $\mu\not\leq\eta$.

\begin{Prop}\label{carc}
For a valuation $\mu \in \mathcal{T}(v, \Lambda)$, we have
$\mathcal{T}_{\mu} = \operatorname{TD}(\mu) \sqcup \{\mathcal{B}_{\nleq}(\mu)\}$.

\end{Prop}
\begin{proof}
Let $\nu \in \mathcal{T} \setminus \{\mu\}$. We consider two cases: $\mu < \nu$ and $\mu \nless \nu$.
If $\mu \nless \nu$, then by Lemma~\ref{ConjB} we have $[\nu]_{\mu} = \mathcal{B}_{\not\leq}(\mu)$.

Now suppose $\mu < \nu$. Take $\omega \in [\nu]_{\mu}$. Then $\mu \notin [\nu, \omega]$. By Lemma~\ref{ConjB}, it follows that $\mu < \omega$. Hence, $\mu \leq \nu \wedge \omega$. If $\mu = \nu \wedge \omega$, this would imply $\mu \in [\nu, \omega]$, contradicting the fact that $\nu \sim_{\mu} \omega$. Therefore, $\mu < \nu \wedge \omega$, and consequently $\nu \wedge \omega \in (\mu, \nu] \cap (\mu, \omega]$. Thus, $\omega \in [\nu]_{\mathrm{tan}}$. Conversely, if $\omega \in [\nu]_{\mathrm{tan}}$, then $\mu < \omega$ and $(\mu, \nu] \cap (\mu, \omega] \neq \emptyset$. From this it follows that $\mu < \nu \wedge \omega$, and hence $\mu \notin [\nu, \omega]$, showing that $\omega \in [\nu]_{\mu}$. Therefore, $[\nu]_{\mu} = [\nu]_{\mathrm{tan}}$.
\end{proof}

By Proposition~\ref{prop:tangent_relation}, the elements of $\mathcal{T}_{\mu}$ that lie above $\mu$ (that is, the classes $[\nu]_{\mu}$ for which $\mu < \omega$ for all $\omega \in [\nu]_{\mu}$) are in bijection with the tangent directions. Besides these elements, there is exactly one additional element in $\mathcal{T}_{\mu}$, namely the set $\mathcal{B}_{\nleq}(\mu)$.

\begin{Obs}\label{obsTD}
Let $\mu \in \mathcal{T}(v, \Lambda)$ be a non-maximal valuation.
Consequently, by \cite[Theorem 1.3 + 1.4]{Nart_1}, we have:
\begin{itemize}
\item if $\mu$ is \textit{value-transcendental}, then $\#\mathrm{TD}(\mu) = 1$;
\item if $\mu$ is \textit{residue-transcendental}, then $\#\mathrm{TD}(\mu) > 1$.
\end{itemize}
With this in mind, observe that if $\mu, \nu \in \mathcal{T}(v, \Lambda)$ are incomparable, then $\mu \wedge \nu$ is residue-transcendental. For this reason, we consider an ordered abelian group $\Lambda$ containing the divisible hull $\Gamma_{\mathbb{Q}}$ in order to guarantee the existence of the $\mu\wedge \nu$ for any two valuations $\mu$ and $\nu$.

By Remark \ref{obs2}, the tree $\mathcal{T}(v, \Gamma_{\mathbb{Q}})$ contains all valuations extending $v$ to $K[x]$, except for the value-transcendental ones. For each element in $\gamma\in\Lambda \setminus \Gamma_{\mathbb{Q}}$, there exists a value-transcendental valuation $\eta$ in $\mathcal{T}(v, \Lambda)$ (which can be constructed through the augmentation process, see (\ref{Vtipos})). However, since $\#\mathrm{TD}(\eta) = 1$, the valuation $\eta$ does not create new ``branches''; that is, the value-transcendental valuations fill the gaps of the tree $\mathcal{T}(v, \Gamma_{\mathbb{Q}})$.
\end{Obs}

\begin{Cor}
Let $\mu \in \mathcal{T}(v, \Lambda)$.
\begin{enumerate}
\item[(a)] If $\mu$ is maximal, then $\#\,\mathcal{T}_{\mu}=1$.
\item[(b)] If $\mu$ is value-transcendental, then $\#\,\mathcal{T}_{\mu}=2$.
\end{enumerate}
\end{Cor}
\begin{proof}
This follows directly from Proposition \ref{carc} and Remark \ref{obsTD}.
\end{proof}

\begin{figure}[htbp]
    \centering
    \begin{subfigure}[b]{0.32\textwidth}
        \centering
        \definecolor{qqqqff}{rgb}{0,0,1}
        \definecolor{sqsqsq}{rgb}{0.12549019607843137,0.12549019607843137,0.12549019607843137}
        \begin{tikzpicture}[line cap=round,line join=round,>=triangle 45,x=3cm,y=2cm]
        \clip(5.3,2.5) rectangle (7.415437932905358,5.309592941055216);
        \draw [line width=0.5pt,color=qqqqff] (6,5)-- (6,3);
        \draw [line width=0.5pt,color=qqqqff] (6.404651750219197,4.0126354806940725)-- (6,3.287880238014136);
        \draw [line width=0.5pt,color=qqqqff] (5.60017,4.83161)-- (6,4.237309605924853);
        \begin{scriptsize}
        \draw [fill=sqsqsq] (6,5) circle (2pt);
        \draw[color=sqsqsq] (6.147167989875226,5.0320178467621455) node {$\mu$};
        \draw[color=qqqqff] (6.129114325205758,4.102254116284541) node {$[\nu]_{\mu}$};
        \end{scriptsize}
        \end{tikzpicture}
        \caption{$\mu$ maximal}
        \label{fig:maximal}
    \end{subfigure}
    \hfill
    \begin{subfigure}[b]{0.32\textwidth}
        \centering
        \definecolor{qqqqff}{rgb}{0,0,1}
        \definecolor{ffqqqq}{rgb}{1,0,0}
        \definecolor{sqsqsq}{rgb}{0.12549019607843137,0.12549019607843137,0.12549019607843137}
        \begin{tikzpicture}[line cap=round,line join=round,>=triangle 45,x=3cm,y=2.5cm]
        \clip(5.6,2.7460187825498554) rectangle (7.329139459741863,5.10856885083845);
        \draw [line width=0.5pt,color=ffqqqq] (6.2,4.8)-- (6.20132898388363,3.86981318366514);
        \draw [line width=0.5pt,color=qqqqff] (6.20132898388363,3.86981318366514)-- (6.2,3.2);
        \draw [line width=0.5pt,color=qqqqff] (6.60302,3.87433)-- (6.200396820690979,3.3999992126635963);
        \draw [line width=0.5pt,color=ffqqqq] (5.84928,4.61001)-- (6.200942068836359,4.140624049354559);
        \begin{scriptsize}
        \draw [fill=sqsqsq] (6.20132898388363,3.86981318366514) circle (2pt);
        \draw[color=sqsqsq] (6.312600059610862,3.9252848455476483) node {$\mu$};
        \draw[color=qqqqff] (6.067505579737379,3.4913470779028044) node {$[\nu]_{\mu}$};
        \draw[color=ffqqqq] (6.4,4.495832651154758) node {$[\eta]_{\mu}$};
        \end{scriptsize}
        \end{tikzpicture}
        \caption{$\mu$ value-transcendental}
        \label{fig:value-trans}
    \end{subfigure}
    \hfill
    \begin{subfigure}[b]{0.32\textwidth}
        \centering
        \definecolor{qqccqq}{rgb}{0,0.8,0}
        \definecolor{yqqqyq}{rgb}{0.5019607843137255,0,0.5019607843137255}
        \definecolor{qqqqff}{rgb}{0,0,1}
        \definecolor{ffqqqq}{rgb}{1,0,0}
        \definecolor{sqsqsq}{rgb}{0.12549019607843137,0.12549019607843137,0.12549019607843137}
        \begin{tikzpicture}[line cap=round,line join=round,>=triangle 45,x=3cm,y=2.5cm]
        \clip(5.5,2.746018782549854) rectangle (7.457713613118121,5.108568850838449);
        \draw [line width=0.5pt,color=ffqqqq] (6.2000976754066395,4.738918159881731)-- (6.204115617699648,4.164352411981613);
        \draw [line width=0.5pt,color=qqqqff] (6.204115617699648,4.164352411981613)-- (6.2,3.2);
        \draw [line width=0.5pt,color=qqqqff] (6.60302,3.87433)-- (6.201228880409683,3.4879455463554807);
        \draw [line width=0.5pt,color=yqqqyq] (6.204115617699648,4.164352411981613)-- (6.6943,4.534);
        \draw [line width=0.5pt,color=qqccqq] (6.204115617699648,4.164352411981613)-- (5.67375,4.54204);
        \begin{scriptsize}
        \draw [fill=sqsqsq] (6.204115617699648,4.164352411981613) circle (2pt);
        \draw[color=sqsqsq] (6.316618001903874,4.138235787077061) node {$\mu$};
        \draw[color=ffqqqq] (6.316618001903874,4.568155612428898) node {$[\eta_2]_{\mu}$};
        \draw[color=qqqqff] (6.067505579737383,3.6400109427440928) node {$[\nu]_{\mu}$};
        \draw[color=yqqqyq] (6.6,4.3) node {$[\eta_3]_{\mu}$};
        \draw[color=qqccqq] (5.987146733877225,4.463689112810695) node {$[\eta_1]_{\mu}$};
        \end{scriptsize}
        \end{tikzpicture}
        \caption{$\mu$ residue-transcendental}
        \label{fig:residue-trans}
    \end{subfigure}
    \caption{Representations of the elements of $\mathcal{T}_{\mu}$ in the valuative tree.}
    \label{fig:valuation_types}
\end{figure}
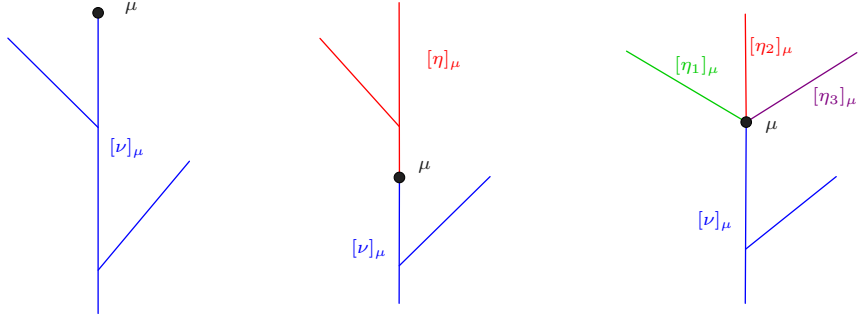

\subsection{The Scott topology} 

Let $(\mathcal{P}, \leq)$ be a partially ordered set.

\begin{Def}

 We say that a non-empty subset $\mathcal{D}$ of $\mathcal{P}$ is a \textbf{directed set} if every pair of elements in $\mathcal{D}$ has an upper bound in $\mathcal{D}$ itself. A subset $\mathcal{O} \subseteq \mathcal{P}$ is \textbf{Scott open} if it satisfies the following conditions:
\begin{itemize}
\item \textbf{Upper Set:}
For all $x, y \in \mathcal{P}$ such that $x \in \mathcal{O}$ and $x\leq y$, we have $y \in \mathcal{O}$ (that is, if an element is in $\mathcal{O}$, then all elements larger than it are also in $\mathcal{O}$).

\item \textbf{Inaccessible by directed joins:} For every directed set $\mathcal{D} \subseteq \mathcal{P}$ that has a supremum in $\mathcal{O}$, it follows that $\mathcal{D}\cap \mathcal{O}\neq \emptyset$.
\end{itemize}
    
\end{Def}

\begin{Def}
The \textbf{Scott topology} $\mathfrak{S}$ on $\mathcal{P}$ is defined as the set of all Scott open sets.
\end{Def}

We will denote by $\mathfrak{S}$ and $\mathfrak{S}_{\mu}$ the Scott topologies of the spaces $(\mathcal{T}, \leq)$ and $(\mathcal{T}, \leq_{\mu})$, respectively.

Our goal is to present a comparison between the weak tree topology and the Scott topology. We begin by proving a relation between the weak tree topology and $\mathfrak{S}_\mu$.

\begin{Lema}\label{lem:superior}
The set $[\nu]_{\mu}$ is an upper set in $(\mathcal{T}, \leq _{\mu})$, for all $\nu \in \mathcal{T} \setminus\{\mu\}$.
\end{Lema}

\begin{proof}
Take elements $\eta, \rho \in \mathcal{T}$ such that $\eta \in [\nu]_{\mu}$ and $\eta < _{\mu} \rho$. We will show that $\rho \in [\nu]_{\mu}$.

Since $\eta <_{\mu} \rho$, we have $[\eta,\rho] = \{\omega \in \mathcal{T} \mid \eta \leq_{\mu} \omega \leq_{\mu} \rho\}$. Since $\mu$ is the smallest element of $(\mathcal{T}, \leq_{\mu})$, we have $\mu \notin [\eta, \rho]$. On the other hand, since $\eta \in [\nu]_{\mu}$, by the definition of $[\nu]_{\mu}$, we have $\mu \notin [\nu, \eta]$. By Lemma \ref{lema1}, we have $[\nu, \rho] \subseteq [\nu, \eta] \cup [\eta, \rho]$. Since $\mu$ does not belong to either set in the union, we conclude that $\mu \notin [\nu, \rho]$. Therefore, $\rho \in [\nu]_{\mu}$.
\end{proof}

\begin{Lema}\label{lem:inacessivel}
 Every sub-basic open set $[\nu]_{\mu}$ in the weak tree topology  is inaccessible by directed joins.
\end{Lema}

\begin{proof}
Consider a sub-basic open set $[\nu]_{\mu}$ of $\mathcal{T}$ and take a directed set $\mathcal{D}$ such that $\mathcal{D} \subseteq \mathcal{T} \setminus [\nu]_{\mu}$. We want to prove that $\sup \mathcal{D} =: \eta \notin [\nu]_{\mu}$. If $\mathcal{D} = \{\mu\}$, then $\sup \mathcal{D} = \mu \notin [\nu]_{\mu}$ and the result follows immediately.

Assume there exists $\rho \in \mathcal{D}$ such that $\rho \neq \mu$.

\noindent \textbf{Case 1:} Suppose first that $\mu \nleq \nu$. We will show that $\mu \leq \mathcal{D}$ (i.e., $\mu \leq \rho$ for all $\rho \in \mathcal{D}$). If there exists $\rho \in \mathcal{D}$ such that $\mu \nleq \rho$, then $\mu \notin [\rho,\nu]$ and $\rho \in [\nu]_{\mu}$, which contradicts $\mathcal{D} \subseteq \mathcal{T} \setminus [\nu]_{\mu}$. Therefore, $\mu \leq \rho$ for all $\rho \in \mathcal{D}$, as we wanted to show. Consequently, $\mu \leq \eta$ and, by Proposition~\ref{carc}, we have $\eta \notin [\nu]_{\mu}$.

\noindent \textbf{Case 2:} Now assume that $\mu \leq \nu$. We want to show that $\nu \wedge \eta \leq \mu$. Indeed, since $\mu \leq \nu$ and $\nu \wedge \eta \leq \nu$, we have $\mu < \nu \wedge \eta$ or $\nu \wedge \eta \leq \mu$, because $(-\infty, \nu]$ is totally ordered. Suppose, by contradiction, that $\mu < \nu \wedge \eta$. We split into two cases:
\begin{itemize}
\item If $\nu \wedge \eta < \eta$, then there exists $\rho \in \mathcal{D}$ such that $\nu \wedge \eta < \rho \leq \eta$. This implies that $\nu \wedge \eta = \nu \wedge \rho$ and consequently $\mu \notin [\nu,\rho]$, contradicting $\rho \notin [\nu]_{\mu}$.

\item If $\nu \wedge \eta = \eta$, then $\eta \leq \nu$. Since $\mu<\nu \wedge \eta \leq \eta$ and $\eta = \sup \mathcal{D}$, there exists $\rho \in \mathcal{D}$ such that $\mu < \rho \leq \eta \leq \nu$ and consequently $\mu \notin [\nu, \rho]$, again a contradiction.
\end{itemize}
Therefore, $\nu \wedge \eta \leq \mu$. Hence, $\mu \in [\nu,\eta]$ and the lemma is proved. 
\end{proof}

\begin{Cor}\label{subbasico_contido_scott}
   For a fixed $\mu \in \mathcal{T}$, $[\nu]_\mu$ is an open set in $\mathfrak{S}_\mu$ for every $\nu \in \mathcal{T}\setminus\{\mu\}$. 
\end{Cor}

The proof of the main theorem in this section relies on the hypothesis that $\Lambda$ is densely ordered. We start with a lemma.



\begin{Lema} \label{lema_sup}
Assume that $\Lambda$ is densely ordered.
\begin{enumerate}
\item[(a)] For all $\mu, \nu \in \mathcal{T}(v, \Lambda)$ with $\mu \neq \nu$, the interval $(\mu, \nu)$ is nonempty.
\item[(b)] For all $\mu \in \mathcal{T}(v, \Lambda)$, we have $\sup (-\infty, \mu)=\mu$.
\end{enumerate}
\end{Lema}

\begin{proof}$\,$
	
	\begin{enumerate}
		\item[(a)] If $\mu$ and $\nu$ are incomparable, then $\mu \wedge \nu \in (\mu, \nu)$, which proves the claim in this case.
		
		Now suppose $\mu$ and $\nu$ are comparable. Without loss of generality, assume $\mu<\nu$. 
		Let $\phi\in\mathbf{t}(\mu,\nu)$; then $\mu(\phi)<\nu(\phi)$. 
		Since $\Lambda$ has a dense order, there exists $\gamma\in\Lambda$ such that 
		$\mu(\phi)<\gamma<\nu(\phi)$. 
		Consider the augmentation
		\[
		\eta=[\mu;\phi,\gamma].
		\]
		We know that $\mu<\eta$. 
		On the other hand, for every $f\in K[x]$ written in its $\phi$-expansion $f=\sum_{j=0}^r f_j\phi^j$ 
		with $\deg(f_j)<\deg(\phi)$ for all $0\leq j\leq r$, we have
		\[
		\eta(f)=\min_{0\leq j\leq r}\{\mu(f_j)+j\gamma\}
		<\min_{0\leq j\leq r}\{\nu(f_j)+j\nu(\phi)\}
		\leq\nu(f),
		\]
		hence $\eta<\nu$. Therefore, $\eta\in(\mu,\nu)$ and the result follows. 
		
		\item[(b)] Suppose that there exists $\omega \in \mathcal{T}$ greater than all elements of $(-\infty, \mu)$, but not comparable to $\mu$. Then the interval $(\mu \wedge \omega, \mu)$ is empty, because if there existed $\rho \in (\mu \wedge \omega, \mu)$, then $\rho \in (-\infty, \mu)$ and $\rho \nleq \omega$, contradicting the choice of $\omega$. However, since $\Lambda$ has a dense order, by the previous item, $(\mu \wedge \omega, \mu)$ cannot be empty. Thus, every upper bound of $(-\infty, \mu)$ must be comparable to $\mu$.
		
		Now, suppose there exists $\eta \in \mathcal{T}$ such that $\nu \leq \eta < \mu$ for all $\nu \in (-\infty, \mu)$. This would imply that $(\eta, \mu) = \emptyset$, which again contradicts the fact that $\Lambda$ is densely ordered. Therefore, $\sup(-\infty, \mu) = \mu$.
	\end{enumerate}

\end{proof}

\begin{Obs}\label{diferencaWS} We can list some differences between the Scott topology and the weak tree topology: 
    
\begin{itemize}
\item  If we consider the orders $\leq_{\mu}$ and $\leq_{\nu}$ on $\mathcal T$ for ${\mu}\neq {\nu}$, then the associated weak tree topologies are the same, but the Scott topologies $\mathfrak S_{\mu}$ and $\mathfrak S_{\nu}$ are not.

\item Suppose $\Lambda$ has a dense order and take $\mu, \nu \in \mathcal{T}$ distinct. If we assume that $\mu<\nu$, then every Scott open set containing $\mu$ also contains $\nu$, since these open sets are upper sets. This shows that the Scott topology is not Hausdorff. In fact, the Scott topology does not satisfies the $T_1$-separation axiom, but it is  $T_0$-separable. On the other hand, by Lemma~\ref{lema_sup}, there exists $\eta\in(\mu,\nu)$, and thus $[\mu]_{\eta}$ and $[\nu]_{\eta}$ are disjoint open sets separating $\mu$ and $\nu$. Therefore, the weak tree topology is Hausdorff.
    \end{itemize}
\end{Obs}

 However, we will be able to describe the weak tree topology in terms of the Scott topology.

\begin{Prop}\label{prop:Scott_mais_fina}
Assume that $\Lambda$ is densely ordered. Every Scott open set $\mathcal{O}$ in $\mathcal{T}(v, \Lambda)$ is open in the weak tree topology.
\end{Prop}

\begin{proof}
For each valuation $\nu \in \mathcal{O}$, we will show that there exists $\mu \in \mathcal{T}$ such that $[\nu]_{\mu} \subseteq \mathcal{O}$.

Consider the set $(-\infty, \nu)$. Then, by Lemma \ref{lema_sup}, $\sup (-\infty, \nu) = \nu \in \mathcal{O}$ and since $\mathcal{O}$ is open in the Scott topology, we have $\mathcal{O} \cap (-\infty, \nu) \neq \emptyset$.

Take any $\mu \in \mathcal{O} \cap (-\infty, \nu)$. For each $\eta \in [\nu]_{\mu}$, we will show that $\eta \in \mathcal{O}$. Suppose, by contradiction, that $\mu \nleq \eta$. Since $\mu \leq \nu$ and $(-\infty, \nu]$ is totally ordered, we have $\mu \leq \nu \wedge \eta$ or $\nu \wedge \eta \leq \mu$.

The first case cannot occur, because $\mu \nleq \eta$ and $\nu \wedge \eta \leq \eta$. Consequently, $\nu \wedge \eta \leq \mu \leq \nu$ and thus $\mu \in [\nu, \eta]$. This means that $\eta \notin [\nu]_{\mu}$, a contradiction. Therefore, $\mu \leq \eta$. Since $\mu \in \mathcal{O}$ and $\mathcal{O}$ is an upper set, we obtain that $\eta \in \mathcal{O}$. Hence, $[\nu]_{\mu} \subseteq \mathcal{O}$, which completes the proof.
\end{proof}

\begin{Prop}\label{propScottCoarser}
Assume that $\Lambda$ is densely ordered. The Scott topology $\mathfrak{S}$ on $\mathcal{T}$ is strictly coarser than the weak tree topology $\mathfrak{W}$.
\end{Prop}

\begin{proof}
Proposition \ref{prop:Scott_mais_fina} shows that the Scott topology $\mathfrak{S}$ is coarser than the weak tree topology $\mathfrak{W}$. To show that the inclusion is strict, take $\mu, \nu \in \mathcal{T}$ two valuations such that $\nu<\mu$. Consider the open set $[\nu]_{\mu}$ in the weak tree topology. Since $\nu \in [\nu]_{\mu}$ and $\nu<\mu$, if $[\nu]_{\mu}$ were an upper set, then $\mu$ should belong to $[\nu]_{\mu}$. However, $\mu \notin [\nu]_{\mu}$. Therefore, $[\nu]_{\mu}$ is not an upper set and consequently not a Scott open set. This shows that the Scott topology is strictly coarser than the weak tree topology.
\end{proof}


\begin{Teo}\label{teoWeakTreeTopGenByScott}
Assume that $\Lambda$ is densely ordered. The weak tree topology coincides with the topology generated by $\bigcup_{\mu\in\mathcal{T}}\mathfrak{S}_{\mu}$.
\end{Teo}

\begin{proof}
By Corollary~\ref{subbasico_contido_scott}, each $[\nu]_{\mu}$ belongs to $\mathfrak{S}_{\mu}$. Conversely, by Proposition \ref{prop:Scott_mais_fina}, every $\mathfrak{S}_{\mu}$-open set is open in the weak tree topology.
\end{proof}

\section{The valuative tree $\mathcal{T}_{\rm sme}$ }\label{SecTreeSME}

\subsection{Construction of $\R^{\mathbb{I}}$, $\R_{\rm sme}$ and $\Gamma_{\rm sme}$. }

In \cite{Nart_1}, it is presented a universal ordered abelian group $\R^{\mathbb{I}}_{\rm lex}$ having the property that each class of valuations whose
restriction to $K$ is equivalent to $v$ contains a valuation $\nu:K[x]\longrightarrow \R^{\mathbb{I}}_{\rm lex}$. We will give a sketch of its construction.

Given $\gamma\in \Gamma$, the \textit{principal convex subgroup} of $\Gamma$ generated by $\gamma$ is the  intersection of all convex subgroups of $\Gamma$ containing $\gamma$. 
Let
$$\mathcal{I}:= {\rm Prin}(\Gamma) $$
be the totally ordered set of non-zero principal convex subgroups of $\Gamma$, ordered by decreasing inclusion.  
We take the Hahn product, indexed by $\mathcal I$, defined as
$$\R_{\text{lex}}^\mathcal{I}:=\{ y=(y_i)_{i\in \mathcal I}\mid \supp(y) \text{ is well-ordered}\} \subset \R^\mathcal{I} $$
where $\supp(y) = \{ i\in \mathcal{I}\mid y_i\neq 0  \}$. Let ${\rm Init}(\mathcal{I})$ be the set of initial segments of $\mathcal{I}$. For each $S\in {\rm Init}(\mathcal{I})$, we consider a formal symbol $i_S$ and the ordered set
$$\mathcal{I}_S:= \mathcal{I}\sqcup\{i_S\}, $$
with the order defined by $i<i_S$ for all $i\in S$ and $i_S<j$ for all $j\in \mathcal{I}\setminus S$. We define
$$\mathbb{I}:= \mathcal{I}\sqcup \{ i_S\mid S\in {\rm Init}(\mathcal{I})\}. $$
The order in $\mathbb{I}$ is such that $\mathcal{I}_S \hookrightarrow \mathbb{I}$ preserves the order for all $S\in  {\rm Init}(\mathcal{I})$ and, for every $S, T\in {\rm Init}(\mathcal{I})$, we have $i_S<i_T$ if and only if $S\subsetneq T$. As above, we consider the Hahn product $$\R_{\text{lex}}^{\mathbb{I}}\subset \R^{\mathbb{I}}. $$ For all $S\in  {\rm Init}(\mathcal{I})$, since $\mathcal{I}\subset \mathcal{I}_S \subset \mathbb{I}$,  we have the embeddings
$$ \R_{\text{lex}}^\mathcal{I} \hookrightarrow \R_{\text{lex}}^{\mathcal{I}_S}\hookrightarrow \R_{\text{lex}}^{\mathbb{I}}$$
Also, it is shown in \cite{Nart_1} that we have an embedding  $$ 
\Gamma \hookrightarrow \Gamma_\Q \hookrightarrow \R_{\text{lex}}^\mathcal{I} \hookrightarrow \R_{\text{lex}}^{\mathbb{I}}.  $$
\begin{Prop}\cite[Proposition 6.2]{Nart_1}
	Let $\mu$ be a valuation on $K[x]$ whose restriction to $K$ is equivalent to $v$. Then, there exists an embedding $j:\Gamma_\mu \hookrightarrow \R_{\rm{lex}}^{\mathbb{I}}$ satisfying the following properties:
	\begin{description}
		\item[(i)] the following diagram commutes:

		\begin{center}
			\begin{tikzcd}[row sep=large, column sep = large]
				K[x] \arrow[r, rightarrow, "j\circ \mu"]{}{}  & (\mathbb{R}_{\rm lex }^{\mathbb{I}})_\infty \\
				K \arrow [u, hookrightarrow ]{}{} \arrow[r , "v", labels=below]{}{} & \Gamma_\infty \arrow [u, hookrightarrow ]{}{}
			\end{tikzcd}
		\end{center}
		
		\item[(ii)] There exists $S\in  {\rm Init}(\mathcal{I})$ such that $j(\Gamma_\mu)\subset \R_{\rm{lex}}^{\mathcal{I}_S}$. 
	\end{description}
\end{Prop}

Hence, for each class of valuations in $ \mathbb{V}$ we can take a representative $\mu$ such that $\mu|_K \sim v$ and $\Gamma_\mu\subset  \R_{\text{lex}}^{\mathcal{I}_S}$ for some $S\in  {\rm Init}(\mathcal{I})$ (\cite{Nart_1}, p.35). We will then assume $\Gamma_v\subset \R_{\rm{lex}}^{\mathcal{I}}$. Consider the tree
$$\mathcal{T}(v,{\R^{\mathbb{I}}}) = \{\nu: K[x] \rightarrow (\R_{\rm{lex}}^{\mathbb{I}})_\infty \mid \nu \text{ is a valuation extending } v  \}. $$

We define\footnote{Here, ``{\rm sme}'' makes reference to the fact that  $\Gamma\hookrightarrow \Gamma_\mu$ is a \textbf{small extension}, meaning that if $\Gamma'\subset \Gamma_\mu$ is the relative divisible closure of $\Gamma$ in $\Gamma_\mu$, then $\Gamma_\mu/\Gamma'$ is a cyclic group (see \cite[Section 6]{Nart_1} for details) }
$$\R_{\text{sme}}:=  \bigcup_{S\in  {\rm Init}(\mathcal{I}) } \R_{\text{lex}}^{\mathcal{I}_S} \subset \R_{\text{lex}}^{\mathbb{I}}.  $$
Then every class of valuations in $\mathbb{V}$ admits a representative $\mu$ such that $\Gamma_\mu\subset \R_{\text{sme}}$.


For each $\alpha\in \R_{\text{sme}}$, let $\langle \Gamma, \alpha \rangle$ be the subgroup generated by $\Gamma$ and $\alpha$. We define the following equivalence relation on $\R_{\text{sme}}$: for $\alpha, \beta\in \R_{\text{sme}}$, 
$ \alpha \sim_{\rm sme} \beta$ if and only if there exists an isomorphism of ordered groups $ \langle \Gamma, \alpha \rangle \to \langle \Gamma, \beta \rangle   $ which sends $\alpha$ to $\beta$ and acts as the identity on $\Gamma$. Take any subset $\Gamma_{\rm sme}\subset \R_{\text{sme}}$ of representatives of the quotient set $\R_{\text{sme}}/\sim_{\rm sme}$. According to \cite{Nart_1}, we have
$$\Gamma \subset \Gamma_\Q \subset \Gamma_{\text{sme}} \subset \R_{\text{sme}}. $$

\subsubsection{Quasi-cuts}
A canonical description for $\Gamma_{\text{sme}}$ is the set of quasi-cuts on $\Gamma_\Q$ (see \cite{Nart_1}).

For $S,T\subset \Gamma_\Q$, we will write $S\leq T$ when   $s\leq t$ for every $s\in S$ and every $t\in T$. A subset $S\subset \Gamma_\Q$ is said to be an \textbf{initial segment} of $\Gamma_\Q$ if  it satisfies the following property: if $\gamma\in \Gamma_\Q$ is such that $\gamma\leq s$ for some $s\in S$, then $\gamma\in S$.

\begin{Def}A \textbf{quasi-cut} in $\Gamma_\Q$ is a pair $\delta = (\delta^L,\delta^R)$ of subsets of $\Gamma_\Q$ such that
	$$\delta^L\leq \delta^R \text{ and } \delta^L\cup \delta^R = \Gamma_\Q. $$
	
\end{Def}

\begin{Obs}We collect bellow some properties of quasi-cuts in $\Gamma_\Q$. 
	
	\begin{itemize}
		\item The subset $\delta^L$ is an initial segment of $\Gamma_\Q$ and $\delta^L\cap \delta^R$ has at most one element. 
		
		\item We denote by ${\rm Qcuts}(\Gamma_\Q)$ the set of all quasi-cuts in $\Gamma_\Q$. We can define a total order in ${\rm Qcuts}(\Gamma_\Q)$ by setting
		$$\delta = (\delta^L,\delta^R)\leq \gamma = (\gamma^L,\gamma^R) \Llr \delta^L\subseteq \gamma^L \text{ and } \delta^R\supseteq \gamma^R. $$
		
		\item There is an embedding $\Gamma_\Q\hookrightarrow {\rm Qcuts}(\Gamma_\Q) $ that preserves the order, mapping $\alpha\in \Gamma_\Q$ to  $(\alpha^L,\alpha^R)$,
		where
		$$\alpha^L = \{ \sigma\in \Gamma_{\Q} \mid \sigma\leq \alpha\} \text{ and } \alpha^R = \{ \sigma\in \Gamma_{\Q} \mid \sigma\geq \alpha \}. $$
		This is called the \textbf{principal quasi-cut} associated to $\alpha$. We use this identification and assume that $\Gamma_\Q \subset {\rm Qcuts}(\Gamma_\Q)$.
		
		\item If $\delta = (\delta^L,\delta^R)$ is such that $\delta^L\cap \delta^R = \emptyset$, then $\delta$ is called a \textbf{cut} in $\Gamma_\Q$. Calling ${\rm Cuts}(\Gamma_\Q)$ the set of all cuts in $\Gamma_\Q$, we have 
		\[
		{\rm Qcuts}(\Gamma_\Q) = \Gamma_\Q \sqcup {\rm Cuts}(\Gamma_\Q).
		\]
		
		
	\item Let $\delta=(\delta^L,\delta^R)$ be a cut in $\Gamma_\Q$. Then $\delta  =\sup \delta^L=\inf \delta^R$.	Indeed, we have $\delta>\alpha=(\alpha^L,\alpha^R)$ for every $\alpha \in  \delta^L$. If $\gamma$ is a quasi-cut such that $\gamma<\delta$, then we can suppose $\gamma^L\subset \delta^L$. Hence, there exists $\alpha\in \delta^L\setminus \gamma^L$ and then  $\alpha>\gamma^L$. That is, $\alpha=(\alpha^L,\alpha^R)>\gamma$ and then $\gamma\neq \sup \delta^L.$ The reasoning is analogous for $\inf\delta^R.$ 

	 \end{itemize}
\end{Obs}


For each $\alpha\in \R_{\text{sme}}$, let $\delta_\alpha$ be the quasi-cut defined as
\[
\delta_\alpha^L = \{ \gamma\in \Gamma_\Q\mid \gamma\leq \alpha  \}\mbox{ and }\delta_\alpha^R = \{ \gamma\in \Gamma_\Q\mid \gamma\geq \alpha  \}.
\]
For all $\alpha,\beta\in \R_{\text{sme}}$, we have $\alpha \sim_{\rm sme } \beta$ if and only if $\delta_\alpha= \delta_\beta$ \cite[Lemma 6.6]{Nart_1}. The map $\alpha	\longmapsto \delta_\alpha$ is an isomorphism of ordered sets between  $\Gamma_{\text{sme}}$ and $\rm Qcuts(\Gamma_\Q)$ (see Section 6.4 of \cite{Nart_1}). Therefore, the set $\Gamma_{\text{sme}}$, when equipped with the order topology, is complete (i.e. every subset of $\Gamma_{\text{sme}}$ has a supremum and an infimum)  and $\Gamma_\Q$ is dense in $\Gamma_{\rm sme}$.

\subsubsection{The subtree $\mathcal{T}_{\rm sme}$}
Consider the subtree
$$\mathcal{T}_{\R_{\text{sme}}}:= \{\nu\in \mathcal{T}(v,{\R^{\mathbb{I}}})\mid \Gamma_\nu \subset \R_{\text{lex}}^{\mathcal{I}_S} \text{ for some } S\in {\rm Init}(\mathcal{I})\}.$$ 

For a valuation $\nu$, let $\phi$  be a key polynomial of minimal degree for $\nu$.  Define $\deg(\nu):=\deg(\phi)$. By \cite[Theorem 3.9]{NartKeyPolyValuedFields},  ${\rm sv}(\nu):=\nu(\phi)$ is well defined.   Consider the subgroup $\Gamma_{\nu}^0=\{\nu(a)\mid 0\leq \deg(a) <\deg(\nu)\}\subset \Gamma_\Q\subset \Gamma_\nu$ \cite[Lemma 2.11]{NartKeyPolyValuedFields}. We have $\Gamma_\nu = \langle \Gamma_\nu^0, {\rm sv}(\nu) \rangle.$

\begin{Prop}\label{PropNartValEquivsme}\cite[Proposition 6.3]{Nart_1}
	Let $\mu, \nu\in \mathcal{T}_{\R_{{\rm sme}}}$. Then $\mu\sim \nu$ if and only if the following three conditions hold:
	\begin{description}
		\item[(a)] The valuations $\mu$ and $\nu$ admit a common  key polynomial of minimal degree.
		
		\item[(b)] For all $a\in K[x]$ such that $\deg(a)<\deg(\mu)$, we have $\mu(a)=\nu(a)$.
		
		\item[(c)] ${\rm sv}(\mu)\sim_{\rm sme } {\rm sv}(\nu)$. 
	\end{description} 
\end{Prop}

\begin{Cor}\label{CorEquivValAreComparable}
	Let $\mu, \nu\in \mathcal{T}_{\R_{{\rm sme}}}$.	If $\mu\sim \nu$, then they are comparable. 
\end{Cor}

\begin{proof}
	By Proposition \ref{PropNartValEquivsme} above, we can find a common key polynomial $\phi$ for $\mu$ and $\nu$ of minimal degree.  By	\cite[Proposition 2.3]{NartJosneiItaliano}, $\mu=\mu_\phi$ and $\nu=\nu_\phi$. Suppose $\mu(\phi)\leq \nu(\phi)$.  Hence, for any $f = f_r\phi^r+\ldots+ f_0$, with $\deg(f_j)<\deg(\phi)$, we have $\mu(f_j)=\nu(f_j)$ for every $j$ and
	$$\mu(f) = \min_{0\leq j\leq r}\{\mu(f_j)+j\mu(\phi)\}\leq  \min_{0\leq j\leq r}\{\nu(f_j)+j\nu(\phi)\}=\nu(f). $$
\end{proof}

Consider now the subtree 
$$\mathcal{T}_{\rm sme}:= \{\nu\in \mathcal{T}_{\R_{ \rm sme}}\mid {\rm sv}(\nu)\in \Gamma_{\rm sme}\}. $$


\begin{Teo}
	\cite[Theorem 7.1]{Nart_1} The mapping $\mu\mapsto [\mu]$ induces a bijection between $\mathcal{T}_{\rm sme}$ and $\mathbb{V}$.
\end{Teo}

	
	
	

\section{Upper bounds and suprema in $\mathcal{T}(v,\Lambda)$}\label{SecUpperBoundsSuprema}

\subsection{Upper bounds of increasing families of valuations}

 Let $\mathfrak{v}=\{\nu_i\}_{i\in I}$ be a totally ordered subset of $\mathcal{T}=\mathcal{T}(v,\Lambda)$, with no maximum. 
Let us assume that $\Gamma_\Q\subseteq \Lambda$.
 We wish to construct upper bounds to $\mathfrak{v}$.

 By \cite[Corollary 2.3]{novbarnabe},  for every $f\in K[x]$, either $\{\nu_i(f)\}_{i\in I}$ is strictly increasing, or there exists $i_0\in I$ such that $\nu_i(f)=\nu_{i_0}(f)$ for every $i\in I$ with $i\geq i_0$. 
%
%
%
 We say that $f$ is $\boldsymbol{\mathfrak{v}}$\textbf{-stable} if there exists $i_f\in I$ such that 
 \begin{equation*}\label{eqVStable}
 	\nu_i(f)=\nu_{i_f}(f) \text{ for every } i\in I \text{ with } i\geq i_f.
 \end{equation*}

We distinguish two cases:

\begin{description}
    \item[Case (1)]  $\mathfrak{v}$ has no maximum and every $f\in K[x]$ is $\mathfrak{v}$-stable;

    \item[Case (2)] $\mathfrak{v}$ has no maximum and there exists $q\in K[x]$ $\mathfrak{v}$-unstable.
\end{description}

A totally ordered subset $\mathfrak{v}$ of $\mathcal{T}(v,\Lambda)$ without maximum will be called an \textbf{increasing family of valuations}. For increasing families in the Case (1), we have the following result.

%

\begin{Prop}\label{lemVstable}\cite[Proposition 4.9]{josneicaio2}
	Let $\mathfrak{v}= \{\nu_i\}_{i\in I}$ be an increasing family in $\mathcal{T}$
	such that 
	every $f\in K[x]$ is $\mathfrak{v}$-stable. 
	Define 
	\begin{displaymath}
	\begin{array}{ccccc}
		\nu& : &K[x]& \longrightarrow &\Lambda_\infty\\
		&&f&\longmapsto &\nu_{i_f}(f)
	\end{array} .	
	\end{displaymath}
	Then $\nu\in \mathcal{T}$ and  $\nu_i< \nu$ for every $i\in I$. Moreover, for $\nu'\in \mathcal{T}$, if $\nu'\leq \nu$ and $\nu_i\leq \nu'$ for every $i\in I$, then  $\nu'=\nu$. 
\end{Prop}

\begin{Obs}
    By \cite[Proposition 4.1]{Nart_1}, the valuation $\nu$ in the above proposition is valuation-algebraic.
\end{Obs}

For the second case, let $\mathfrak{v}= \{\nu_i\}_{i\in I}$ be an increasing family in $\mathcal{T}(v,\Lambda)$ such that  there is at least one polynomial that is  $\mathfrak{v}$-unstable. 
For every $\mathfrak{v}$-stable $f$ we set $\mathfrak{v}(f) = \nu_{i_f}(f)$. Let $Q$ be a monic $\mathfrak{v}$-unstable polynomial  of smallest  degree among all $\mathfrak{v}$-unstable polynomials. A polynomial with this property is called a \textbf{limit key polynomial} for $\mathfrak{v}$ (see \cite{Nart_1} and \cite{NartJosneiItaliano}).
Take $\gamma\in \Lambda_\infty$ such that $\gamma>\nu_i(Q)$ for every $i\in I$. 
Consider the map 
$$\mu_{Q,\gamma}(f)=\underset{0\leq j \leq r}{\min}\{ \mathfrak{v}(f_j)+j\gamma  \}, $$
where $f_0+f_1Q+\ldots+f_rQ^r$ is the $Q$-expansion of $f$. The following result is, for instance, a consequence of  \cite[Proposition 4.6]{Nart_1} or \cite[Theorem 2.4]{novbarnabe}.
\begin{Prop}\label{propMuQGamma}
	We have $\mu_{Q,\gamma}\in \mathcal{T}$ and  $\nu_i<\mu_{Q,\gamma}$ for every $i\in I$. Moreover, $Q$ is a key polynomial for $\mu_{Q,\gamma}$ of minimal degree.
\end{Prop}

%
%

\subsection{Supremum of an increasing family of valuations}\label{SecSupremum}

\begin{Def}
	Let $\mathfrak{v}= \{\nu_i\}_{i\in I}$ be a totally ordered set in $\mathcal{T}$ and $\mu\in \mathcal{T}$. We say that $\mu$ is a \textbf{supremum} of $\mathfrak{v}$ if $\nu_i\leq \mu$ for every $i\in I$ and if $\mu'\in \mathcal{T}$ is such that $\nu_i\leq \mu'$ for every $i\in I$ and $\mu'\leq \mu$, then  $\mu=\mu'$.  
\end{Def}

Even though we have a partial order in $\mathcal{T}(v,\Lambda)$, we will prove that the supremum of an increasing family of valuations is unique. 
For this, we will use Proposition \ref{LemExisteJoint}.
%
%
%

\begin{Lema}
	The supremum of an increasing family $\mathfrak{v}\subset \mathcal{T}(v,\Lambda)$ is unique. 
\end{Lema}

\begin{proof}Take $\mu_1,\mu_2\in \mathcal{T}$ suprema of $\mathfrak{v}=\{\nu_i\}_{i\in I}$. We first see that they are comparable valuations. Suppose, aiming for a contradiction, that $\mu_1\not\leq \mu_2$ and $\mu_2\not\leq \mu_1$. By Lemma \ref{LemExisteJoint}, there exists  $\mu_1\wedge\mu_2$ and, since it is residue-transcendental, we have $\mu_1\wedge\mu_2\in \mathcal{T}$. Since $\mu_1$ and $\mu_2$ are incomparable, we must have $\mu_1\wedge\mu_2<\mu_1$ and $\mu_1\wedge\mu_2<\mu_2$. Since the set $(-\infty, \mu_1]$ is totally ordered \cite[Theorem 2.4]{Nart_2}, $\mu_1\wedge\mu_2$ and $\nu_i$ are comparable for each $i\in I$. We have two cases.
	
	\begin{itemize}
		\item If $ \mu_1\wedge\mu_2\geq \nu_i$ for every $i\in I$, then using that $\mu_1$ and $\mu_2$ are suprema we conclude that $\mu_1=\mu_1\wedge\mu_2=\mu_2$, a contradiction. 
		
		\item If $\mu_1\wedge\mu_2< \nu_j$ for some $j\in I$, then we have $\mu_1\wedge\mu_2<\nu_j\leq \mu_1$ and $\mu_1\wedge\mu_2<\nu_j\leq \mu_2$, contradicting the definition of $\mu_1\wedge\mu_2$.
	\end{itemize}
	
	Hence, $\mu_1$ and $\mu_2$ are comparable, say $\mu_1\leq \mu_2$. Since $\mu_2$ is a supremum to $\mathfrak{v}$, we conclude that $\mu_1=\mu_2$.
\end{proof}

We will denote by $\sup_{i\in I}\nu_i$ the unique supremum of the family $\mathfrak{v}$.


\begin{Obs}
	If every polynomial is $\mathfrak v$-stable, then the valuation $\nu$ constructed in Proposition \ref{lemVstable} is in fact a supremum for the family $\mathfrak{v}$. 
\end{Obs}

    For families in Case (2) we will be able to construct a canonical supremum for $\mathfrak{v}$ when we look to the special tree $\mathcal{T}_{\rm sme}$.
    

\subsection{Supremum of a family with $\mathfrak{v}$-unstable polynomials}


Let $\mathfrak{v}= \{\nu_i\}_{i\in I}$ be an increasing family in $\mathcal{T}_{\rm sme}$ that admits $\mathfrak{v}$-unstable polynomials. 
Let $Q$ be a limit key polynomial for $\mathfrak{v}$. 
We saw that for $\gamma\in \R^\mathbb{I}_\infty$ such that $\gamma>\nu_i(Q)$ for every $i\in I$ we have a valuation  given by Proposition \ref{propMuQGamma}:
$$\mu_{Q,\gamma}(f)=\underset{0\leq j \leq r}{\min}\{ \mathfrak{v}(f_j)+j\gamma  \}, $$
where $f_0+f_1Q+\ldots+f_rQ^r$ is the $Q$-expansion of $f$. 
%
Take the element $\gamma=\sup\gamma^L\in (\Gamma_{\rm sme})_\infty$ that defines the cut in $\Gamma_\Q$ with
$$\gamma^L=\{\alpha\in \Gamma_{\Q }\mid \alpha \leq \nu_i(Q) \text{ for some } i \in I  \} .$$

\begin{Lema}\label{lemSupAlgType}
	The cut $\gamma$ and the valuation $\mu_{Q,\gamma}$ are independent of the choice of the polynomial $Q$, among the monic $\mathfrak v$-unstable polynomials of smallest degree. Moreover, $\mu_{Q.\gamma}\in \mathcal{T}_{\rm sme}$. 
    

\end{Lema}

\begin{proof}Suppose  that $Q'$ is another polynomial of minimal degree not $\mathfrak{v}$-stable and write $Q'=Q+h$. Since $\deg(Q)=\deg(Q')$ and both are monic, we have $\deg(h)<\deg(Q)$, then $h$ is $\mathfrak{v}$-stable.  Since $\{\nu_i(h)\}_{i\in I}$ is ultimately constant and $\{\nu_i(Q)\}_{i\in I}$ and $\{\nu_i(Q')\}_{i\in I}$ are increasing, we deduce that $\nu_i(Q)=\nu_i(Q')$ for sufficiently large $i\in I$. This implies that $\gamma$ does not depend on the choice of $Q$ among the monic $\mathfrak v$-unstable polynomials of smallest degree and $\mu_{Q,\gamma}=\mu_{Q',\gamma}$.

    For the second part, by the definition of $\mu_{Q,\gamma}$ we know that $Q$ is a key polynomial of minimal degree for this valuation (see \cite[Proposition 4.2]{josneicaio2} and \cite[Corollary 7.13]{NartKeyPolyValuedFields}). Hence, ${\rm sv}(\mu_{Q,\gamma})=\mu_{Q, \gamma}(Q)=\gamma\in \Gamma_{\rm sme}.$ Therefore, $\mu_{Q.\gamma}\in \mathcal{T}_{\rm sme}$.
    \end{proof}

\begin{Teo}\label{teoSupinTsme}
    Every increasing family of valuations in $\mathcal{T}_{\rm sme}$ admits a supremum. 
\end{Teo}

\begin{proof}
If $\mathfrak{v}\subset \mathcal{T}_{\rm sme}$ is such that every polynomial is $\mathfrak{v}$-stable, then the supremum was constructed in Proposition \ref{lemVstable}. 

	Let $\mathfrak{v}= \{\nu_i\}_{i\in I}$ be an increasing family in $\mathcal{T}_{\rm sme}$ which admits $\mathfrak{v}$-unstable polynomials. Take $Q$ and $\gamma$ as in Lemma \ref{lemSupAlgType} above. We will show that $\mu_{Q, \gamma}$ is a supremum for $\mathfrak{v}$. Suppose $\mu\in \mathcal{T}_{\rm sme}$ is such that $\nu_i\leq \mu$ for every $i\in I$ and $\mu\leq \mu_{Q,\gamma}$. 	
	For $f\in K[x]$, if $\deg(f)<\deg(Q)$, then $f$ is  $\mathfrak{v}$-stable and for some $i\in I$ we have
	\[
	\mu_{Q,\gamma}(f)=\nu_i(f)\leq \mu(f)\leq \mu_{Q,\gamma}(f), 
	\]
	hence $\mu(f) = \mu_{Q,\gamma}(f)$.
	
we have $Q$ is a key polynomial of minimal degree for $\mu_{Q,\gamma}$. We will show that $Q$ is also a key polynomial for $\mu$ of minimal degree.	Suppose $\mu(Q)\in \Gamma_\Q$. By hypothesis $\gamma=\mu_{Q,\gamma}(Q)\geq \mu(Q)$. Since $\gamma = \sup\{\nu_i(Q)\}$ and $\mu(Q)\in \Gamma_\Q$, we must have some $i\in I$ such that $\gamma>\nu_i(Q)\geq \mu(Q)$. However, by hypothesis $\mu\geq \nu_i$, hence $\mu(Q)=\nu_i(Q)$. Since $Q$ is unstable, we can take $j>i$ such that $\nu_j(Q)>\nu_i(Q)=\mu(Q)$, contradicting the fact that  $\mu\geq \nu_j$ for every $j$.  Hence,  $\mu(Q)\not\in \Gamma_\Q$.
	We conclude that $Q$ has minimal degree among the polynomials $g$ such that $\mu(g)\not \in \Gamma_\Q$. By \cite[Theorem 4.2]{NartKeyPolyValuedFields}, $Q$ is a key polynomial of minimal degree for $\mu$.  Moreover, $\mu(f)=\min_{0\leq j \leq r}\{\mu(f_j)+j\mu(Q)\}$ for every $f\in K[x]$.
	
	We also note that $\mu(Q)$ and $\gamma$ define the same cut in $\Gamma_\Q$. Indeed, 
	$$\alpha\in \gamma^L\Longrightarrow \alpha\leq \nu_i(Q)\leq \mu(Q) \Longrightarrow \alpha \in \mu(Q)^L $$
	and 
	$$\alpha \in \mu(Q)^L \Longrightarrow \alpha<\mu(Q)\leq \mu_{Q,\gamma}(Q)=\gamma \Longrightarrow \alpha \in \gamma^L. $$
	By \cite[Lemma 6.6]{Nart_1}, this means that $\gamma\sim_{\rm sme} \mu(Q)$. Since $\Gamma_{\rm sme}$ is defined taking only one representative for each class of the equivalence relation $\sim_{\rm sme}$, we must have $\gamma= \mu(Q)$. Hence, $\mu=\mu_{Q,\gamma}$.
	
\end{proof}

\subsubsection{Suprema in $\mathcal{T}_{\R_{{\rm sme}}}$}

Suppose we are working with families $\mathfrak{v}= \{\nu_i\}_{i\in I}$  in $$\mathcal{T}_{\R_{{\rm sme}}} := \{\nu\in \mathcal{T}(v,\R^{\mathbb{I}})\mid  \Gamma_\nu\subset \R^{\mathcal{I}_S}_{\rm lex}\text{ for some } S\in {\rm Init}(\mathcal{I}) \text{ and }\nu|_K=v  \}. $$

If all polynomials are $\mathfrak{v}$-stable, then the valuation $\sup \mathfrak{v}$ from Proposition \ref{lemVstable}  still works as a supremum for $\mathfrak{v}$. However, in the case where we can take $Q$ $\mathfrak{v}$-unstable with minimal degree, even if we took $\gamma\in (\Gamma_{\rm sme })_\infty$ as in Lemma \ref{lemSupAlgType}, it may happen that $\mu_{Q,\gamma}$ is not a supremum for the family. Indeed, in the proof of Theorem  \ref{teoSupinTsme}, we see that if a valuation $\mu\in \mathcal{T}_{\R_{{\rm sme}}}$ is such that $\mu\geq \nu_i$ for every $i\in I$ and $\mu\leq \mu_{Q, \gamma}$, then $\mu\sim \mu_{Q, \gamma}$. Since there may exist several representatives for a given equivalence class under the relation $\sim_{\rm sme}$,  it suggests that there could be a family $\{\nu_i\}_{i\in I}$  admitting in $\mathcal{T}_{\R_{{\rm sme}}}$  a strictly decreasing set of upper bounds without minimum element. 

Therefore, in $\mathcal{T}_{\R_{{\rm sme}}}$ we may consider an alternative definition for the supremum of an increasing family of valuations. 
\begin{Def}
	Let $\mathfrak{v}= \{\nu_i\}_{i\in I}$ be an increasing family of valuations  and $\mu\in\mathcal{T}_{\R_{{\rm sme}}}$. We say that $\mu$ is a $\boldsymbol{{\rm sme}}$\textbf{-supremum} of $\mathfrak{v}$ if $\nu_i\leq \mu$ for every $i\in I$ and if $\mu'\in \mathcal{T}_{\R_{{\rm sme}}}$  is such that $\nu_i\leq \mu'$ for every $i\in I$ and $\mu'\leq \mu$, then  $\mu\sim\mu'$. 
\end{Def}

\begin{Obs}
	If $\mu$ is the supremum for $\mathfrak{v}$, then it is a {\rm sme}-supremum for $\mathfrak{v}$. In particular, if all polynomials are $\mathfrak{v}$-stable, then the {\rm sme}-supremum of $\mathfrak{v}$ is the traditional supremum and hence is unique.
\end{Obs}

%
%
\begin{Exa}
	In the conditions of Lemma \ref{lemSupAlgType}, $\mu_{Q,\gamma}$ is a {\rm sme}-supremum for the family $\mathfrak{v}$. 
\end{Exa}

 We will prove that the {\rm sme}-supremum is unique up to equivalence of valuations.  Therefore, for an increasing family $\mathfrak{v}$ with $\mathfrak{v}$-unstable polynomials, the valuation $\mu_{Q,\gamma}$ from Lemma \ref{lemSupAlgType} is a canonical choice of representative for the class of {\rm sme}-suprema of $\mathfrak{v}$.  

\begin{Prop} For an increasing family $\mathfrak{v}\subset \mathcal{T}_{\R_{{\rm sme}}}$,	all  {\rm sme}-suprema for $\mathfrak{v}$ are equivalent. In particular, they are comparable.

\end{Prop}
\begin{proof}		If all polynomials are $\mathfrak{v}$-stable, then the result follows. 
	Suppose $Q$ is a polynomial $\mathfrak{v}$-unstable of minimal degree.  Let $\mu_1$ and $\mu_2$ be {\rm sme}-suprema for $\mathfrak{v}$.
	Without lost of generality, suppose $\mu_1(Q)\leq \mu_2(Q)$. By \cite[Lemma 4.5]{josneicaio2}, $\nu_i(Q)<\mu_1(Q)$ and $\nu_i(Q)<\mu_2(Q)$ for every $i\in I$. We use Proposition \ref{propMuQGamma} to construct two valuations:
	$$\mu_{Q,\mu_1(Q)} \text{ and } \mu_{Q,\mu_2(Q)}. $$
	Both $\mu_{Q,\mu_1(Q)}$ and $\mu_{Q,\mu_2(Q)}$ are upper bounds to $\mathfrak{v}$.  By \cite[Lemma 4.5]{josneicaio2}, we conclude that $$\mu_1(f)=\mu_2(f)=\mu_{Q,\mu_1(Q)}(f)=\mu_{Q,\mu_2(Q)}(f)=\nu_{i_f}(f)$$ for polynomials $f$ with $\deg(f)<\deg(Q)$, since they are $\mathfrak{v}$-stable.  Then,	looking to the action on $Q$-expansions, we see that 
	$$\mu_1\geq \mu_{Q,\mu_1(Q)} \text{ and } \mu_2\geq  \mu_{Q,\mu_2(Q)} \geq \mu_{Q,\mu_1(Q)}.$$
	Since $\mu_1$ and $\mu_2$ are {\rm sme}-suprema, we must have $\mu_1\sim  \mu_{Q,\mu_1(Q)}$ and $\mu_2\sim \mu_{Q,\mu_1(Q)}$. By transitivity and reflexivity of the equivalence relation, we obtain that $\mu_1\sim \mu_2$. By Corollary \ref{CorEquivValAreComparable}, we have $\mu_1\leq \mu_2$.

\end{proof}

\subsection{Suprema and limits of increasing families of valuations}

Let $(Y, \tau)$ be a topological space and $I$ a directed set. Consider $\{y_i\}_{i\in I}\subset Y$ a subset indexed by $I$. We say that a point $y\in Y$ is a \textbf{limit} for $\{y_i\}_{i\in I}$ in $Y$   if for every open neighborhood $U$ of $y$ there exists $i_y\in I$ such that $y_j\in U$ for every $j\geq i_y$, $j\in I$. We will use the  notation $\lim_{i\in I}y_i \to y$ in general and  $\lim_{i\in I}y_i = y$ when $y$ is the only limit for $\{y_i\}_{i\in I}$ (which happens, for example, when $\tau$ is Hausdorff).

For $\mathfrak{v}=\{\nu_i\}_{i\in I}\subset \mathcal{T}(v,\Lambda)$ admitting a supremum, we want to see  if $\lim_{i\in I}\nu_i$ exists and its relationship to $\sup_{i\in I}\nu_i$.
If $ \mathcal{T}$ is endowed with the Scott topology, then we have the following  immediate result.

\begin{Prop}\label{PropSupLimScott}
	We have $\lim_{i\in I}\nu_i \to \sup_{i\in I}\nu_i $  when we consider $\mathcal{T}$ with the Scott topology.

	\end{Prop}

\begin{proof}Take $U$ any open set in the Scott topology such that $\sup\mathfrak{v}\in U$. Since $\{\nu_i\}_{i\in I}$ is a directed set (because it is totally ordered) with its supremum in $U$, by the definition of Scott open set we must have $\{\nu_i\}_{i\in I}\cap U\neq \emptyset$. Since $U$ is an upper set, it follows that $\lim_{i\in I}\nu_i \to \sup_{i\in I}\nu_i $.
	
	\end{proof}

A stronger result is true for the  weak tree topology in $\mathcal{T}$, as we see in the following.

\begin{Lema}
    Suppose $\lim_{i\in I}\nu_i\to \nu$ when we consider $\mathcal{T}$ with the weak tree topology. Then, $\nu \geq \nu_i$ for every $i\in I$.
\end{Lema}

\begin{proof}Suppose there exists $j\in I$ such that either $\nu$ and $\nu_j$ are incomparable or $\nu<\nu_j$. In both cases, consider the open sub-basic set $[\nu]_{\nu_j}.$ We see that $\nu\in[\nu]_{\nu_j}.$  
\begin{itemize}
    \item If $\nu$ and $\nu_j$ are incomparable, then $\nu_i\not\in [\nu]_{\nu_j}$ for every $i>j$, since $\nu_i>\nu_j$ implies that $\nu\wedge \nu_i=\nu\wedge \nu_j$ and then $\nu_j\in [\nu,\nu_i].$

    \item If $\nu<\nu_j$, then also $\nu_i\not\in [\nu]_{\nu_j}$ for every $i>j$, since $\nu_j\in [\nu,\nu_i]. $
\end{itemize}

Hence, in both cases, we contradict the assumption $\lim_{i\in I}\nu_i\to \nu$. Therefore, $\nu$ is an upper bound of $\{\nu_i\}_{i\in I}.$

\end{proof}

\begin{Prop}\label{propLimToSup}
    We have $\lim_{i\in I}\nu_i = \sup_{i\in I}\nu_i $  when we consider $\mathcal{T}$ with the weak tree topology. 
\end{Prop}

\begin{proof}We first prove that $\lim_{i\in I}\nu_i \to \sup_{i\in I}\nu_i$. 
Let $[\nu]_\mu$ be a sub-basic open set and suppose $\sup_{i\in I}\nu_i\in [\nu]_\mu$. By Proposition \ref{carc}, we consider two cases.

If $[\nu]_{\mu}=[\nu]_{\text{tan}}$, then $\sup_{i\in I}\nu_i>\mu$. Hence, there exists $j \in I$ such that $\mu\leq\nu_j<\sup_{i\in I}\nu_i$. Thus, since for all $i>j$ we have $\mu\leq\nu_j<\nu_i<\sup_{i\in I}\nu_i$, it follows that $\nu_i\in[\sup_{i\in I}\nu_i]_{\text{tan}}=[\nu]_{\text{tan}}=[\nu]_\mu$ for all $i>j$.
        
If $[\nu]_\mu=\mathcal{B}_{\nleq}(\mu)$, then $\mu \not\leq\sup_{i\in I}\nu_i$. We have $\mu\nleq\nu_i$ for all $i \in I$. Indeed, if there existed some $i \in I$ such that  $\mu\leq\nu_i$, then we would have $\mu\leq\nu_i<\sup_{i\in I}\nu_i$, which cannot happen. Therefore, $\nu_i\in \mathcal{B}_{\nleq}(\mu)=[\nu]_\mu$ for all $i \in I$.

Thus, $\lim_{i\in I}\nu_i \to \sup_{i\in I}\nu_i$. For the uniqueness of the limit, suppose $\nu\in \mathcal{T}$ is such that $\lim_{i\in I}\nu_i \to \nu.$ 

As we saw in the above lemma,  $\nu_i\leq \nu$ for  every $i\in I$. Suppose $\mu\in \mathcal{T}$ is such that $\nu_i\leq \mu\leq \nu$ for every $i\in I$. If $\mu<\nu$, then when we consider the open set $[\nu]_\mu$ we would have $\nu\in [\nu]_\mu $ and $\nu_i\not\in [\nu]_\mu$ for every $i\in I$. This contradicts $\lim_{i\in I}\nu_i \to \nu.$ Hence, $\mu=\nu$ and we conclude that $\nu=\sup_{i\in I}\nu_i$. 
Therefore,  $\lim_{i\in I}\nu_i = \sup_{i\in I}\nu_i$.

 
\end{proof}

\section{The tree $\mathcal{T}(v,\Lambda)$ as a topological subspace of the product $(\Lambda_\infty)^{K[x]}$}\label{SecClosenessCriterion}

\subsection{A closeness criterion}

Let $R$ be a fixed ring and $\Lambda$ a fixed ordered abelian group. 
Consider the product topology on $(\Lambda_{\infty})^{R}$ induced by a given topology on $\Lambda_{\infty}$. 
The set $\mathcal{V}_{R}(\Lambda)$ of all valuations on $R$ taking values in $\Lambda_{\infty}$ is a subset of $(\Lambda_{\infty})^{R}$.  The following result provides a criterion for determining whether $\mathcal{V}_{R}(\Lambda)$ is closed in $(\Lambda_{\infty})^{R}$.   

\begin{Teo} \label{criterio}
	Let $\Lambda'$ be any submonoid of $\Lambda_{\infty}$ and take a topology $\mathfrak{U}$ on $\Lambda'$ such that
	\begin{description}
		\item [(P1)] the addition $+ : \Lambda'\times\Lambda'\longrightarrow\Lambda'$ is continuous, and
		\item [(P2)] for every $\gamma, \gamma' \in \Lambda'$ such that $\gamma<\gamma'$, there exist open sets $U, U' \in \mathfrak{U}$, such that $\gamma \in U, \gamma' \in U'$ and $U < U'$ (i.e., $u<u'$ for every $u \in U$ and $u'\in U$).
	\end{description}
	Then the set $\mathcal{V}_{R}(\Lambda')$ of valuations of $R$ taking values in $\Lambda'$ is closed in $(\Lambda')^R$.
\end{Teo}

\begin{proof}
	We will prove that $(\Lambda_{\infty})^R\setminus\mathcal{V}_{R}(\Lambda')$ is an open set in the product topology. Take a function $f:R\longrightarrow\Lambda_{\infty}$ which is not a valuation. Then one of the three axioms \textbf{(V1)}, \textbf{(V2)} or \textbf{(V3)} does not hold for $f$. We will treat each case separately. To do this, consider the projection $\pi_a: (\Lambda')^R\longrightarrow\Lambda'$ given by $\pi_a(f):=f(a)$ for all $a \in R$, which is naturally continuous in the product topology.
	
	If $\textbf{(V1)}$ does not hold, then $f(ab)\neq f(a) + f(b)$ for some $a, b \in R$. The property \textbf{(P2)} implies that $\mathfrak U$ is Hausdorff, so there exist $U,W\in\mathfrak U$ such that $f(a)+f(b)\in U$, $f(ab)\in W$ and $U\cap W=\emptyset$. By \textbf{(P1)} there exist $V,V'\in\mathfrak U$ with $f(a)\in V$ and $f(b)\in V'$ such that $V+V'\subseteq U$. Take the open set given by
	\[
	O=\pi_a^{-1}(V)\cap\pi_b^{-1}(V')\cap\pi_{ab}^{-1}(W).
	\]
	Clearly $f\in O$. Take any element $g\in O$ and let's prove that $g$ is not a valuation. Since $g(a)\in V$ and $g(b)\in V'$ we must have $g(a)+g(b)\in V+V'\subseteq U$. Also, $g(ab)\in W$ which means that $g(ab)\neq g(a)+g(b)$ because $U\cap W=\emptyset$. Hence, $g$ is not a valuation.
	
	If \textbf{(V2)} does not hold, then $f(a+b)<\min\{f(a), f(b)\}$ for some $a, b \in R$. In this case we have $f(a+b)< f(a)$ and $f(a+b)< f(b)$. By property \textbf{(P2)} we have there exist open sets $U,U',W,W'\in\mathfrak U$ such that $U < W, U' < W', f(a)\in W, f(b)\in W'$ and $f(a+b)\in U\cap U'$. Take now
	\[
	O=\pi_a^{-1}(W)\cap\pi_b^{-1}(W')\cap\pi_{a+b}^{-1}(U\cap U').
	\]
	Again we have $f\in O$. If $g\in O$ we have $g(a+b) <\min\{g(a),g(b)\}$ which means that $g$ is not a valuation.
	
	Finally, if \textbf{(V3)} does not hold, then $f(1)\neq 0$ or $f(0)\neq \infty$. Assume that $f(1)\neq 0$. Since $\mathfrak{U}$ is Hausdorff the set $\Lambda'\setminus \{0\}$ is open. Take the set $O=\pi_{1}^{-1}(\Lambda'\setminus\{0\})$. Then $f \in O$ and $O\cap \mathcal{V}_{R}(\Lambda')=\emptyset$. The case of $f(0)\neq\infty$ is treated analogously.
\end{proof}

\subsection{The induced order topology in $\mathcal{V}_R(\Lambda)$}


Since $\Lambda$ is a totally ordered set, it carries the order topology. Among the various natural extensions of this topology to $\Lambda_\infty$, we will focus on one in particular.

\begin{Def}
	For a totally ordered set $Y$, the \textbf{order topology} is defined as the topology generated by the sets of the form
	$$\{y \in Y\mid y>y_0\}\text{ and }\{y \in Y{ }|\text{ }y<y_0\}$$
	where $y_0$ runs through $Y$. We denote by $Y_{\infty}$ the set $Y\cup\{\infty\}$ where $\infty$ is an element not belonging to $Y$ and extend the order from $Y$ to $Y_{\infty}$ by setting $\infty>y$ for every $y \in Y$. In this manner, $Y_{\infty}$ is a totally ordered set and hence we can talk about the order topology on $Y_{\infty}$. A neighbourhood basis of $\infty$ in this topology is given by
\[
\{y \in Y_{\infty} \mid y > y_0\},
\]
with $y_0$ running through $Y$.
\end{Def}


\begin{Lema} \label{order}
	The properties \textbf{(P1)} and \textbf{(P2)} hold for the order topology.
\end{Lema}

\begin{proof}
	Take $\gamma,\gamma'\in \Lambda$ such that $\gamma<\gamma'$. If there is an element $\alpha\in(\gamma,\gamma')$ we take
	\[
	U=(-\infty,\alpha)\textrm{ and }U'=(\alpha,\infty).
	\]
	If $(\gamma,\gamma')=\emptyset$ we take
	\[
	U=(-\infty,\gamma')\textrm{ and }U'=(\gamma,\infty).
	\]
	In each case we have $\gamma\in U<U'\ni\gamma'$. Therefore, \textbf{(P2)} holds for the order topology.
	
	In order to have \textbf{(P1)} we must show that for any $\gamma,\gamma'\in\Lambda_\infty$ and $U$ is an open set in the order topology, if $\gamma+\gamma'\in U$ then there exist open sets $V$ and $V'$ with $\gamma\in V$ and $\gamma'\in V'$ such that $V+V'\subseteq U$.
	
	First, consider the case where $\gamma\neq\infty\neq\gamma'$. If the order topology is discrete we just take $V=\{\gamma\}$ and $V'=\{\gamma'\}$. On the other case, take $\alpha,\beta\in\Lambda$ with $\alpha,\beta>0$ such that
	\[
	\gamma+\gamma'\in (\gamma+\gamma'-\alpha,\gamma+\gamma'+\beta)\subseteq U.
	\]
	There exist $\alpha_1,\alpha_2,\beta_1,\beta_2\in\Lambda$ such that
	\[
	\alpha_1,\alpha_2,\beta_1,\beta_2>0\textrm{ and }\alpha_1+\alpha_2=\alpha\textrm{ and }\beta_1+\beta_2=\beta.
	\]
	Consider now the open sets
	\[
	V=(\gamma-\alpha_1,\gamma+\beta_1)\textrm{ and }V'=(\gamma'-\alpha_2,\gamma'+\beta_2).
	\]
	Therefore,
	\[
	V+V'\subseteq (\gamma+\gamma'-\alpha_1-\alpha_2,\gamma+\gamma'+\beta_1+\beta_2)\subseteq U.
	\]
	
	It remains to prove that given any open neighbourhood $U$ of $\infty$ and any $\gamma\in\Lambda_{\infty}$ then there exist open sets $V$ and $V'$ with $\infty\in V$ and $\gamma\in V'$ such that $V+V'\in U$. Since $U$ is a neighbourhood of $\infty$ there exists $\alpha\in \Lambda$ such that $\{\alpha'\in\Lambda_\infty\mid \alpha'>\alpha\}\subseteq U$. If $\gamma=\infty$ we just take
	\[
	V=\{\alpha'\in\Lambda_\infty\mid\alpha'>\alpha\} \textnormal{ and }V'=\{\alpha'\in\Lambda_\infty\mid\alpha'>0\}
	\]
	and if $\gamma\neq\infty$ we just take any $\beta>0$ and define $V=\{\alpha'\in\Lambda_\infty\mid\alpha'>\alpha-\gamma+\beta\}$ and $V'=\{\alpha'\in\Lambda_\infty\mid\alpha'>\gamma-\beta\}$. In any case we have $V+V'\subseteq U$.
\end{proof}

As a consequence of Theorem \ref{criterio} and Lemma \ref{order} we obtain:

\begin{Cor}\label{fechado}
	The set $\mathcal{V}_{R}(\Lambda')$ is closed in $(\Lambda_{\infty})^R$ if we take the order topology on $\Lambda_{\infty}$.
\end{Cor}

\subsection{The tree $\mathcal{T}$ is closed in $(\Lambda_\infty)^{K[x]}$}

Consider the tree $\mathcal{T}=\mathcal{T}(v, \Lambda')$, where $\Lambda'$ is a monoid. Then, $\mathcal{T}$ is a subset of $(\Lambda _{\infty})^{K[x]}$, where $\Lambda'\hookrightarrow \Lambda$ and $\Lambda$ is an ordered abelian group. 



In this context, we have the following result:

\begin{Prop}\label{Tfechado}
	Let $\Lambda'$ be any submonoid of $\Lambda_{\infty}$ and take a topology $\mathfrak{U}$ on $\Lambda'$ satisfying properties \textbf{(P1)} and \textbf{(P2)}. Then, the tree $\mathcal{T}$ is a closed set in $(\Lambda_{\infty})^{K[x]}$. 
\end{Prop}

\begin{proof}
	Consider the following subset of $(\Lambda_{\infty})^{K[x]}$
	\[
	S := \{ f \in (\Lambda_{\infty})^{K[x]} \mid f(a) = v(a) \text{ for all } a \in K \}.
	\]
	We will show that $S$ is closed in $(\Lambda_{\infty})^{K[x]}$ with respect to the topology $\mathfrak{U}$ by proving that its complement $S^c$ is open. Note that
	\[
	\begin{aligned}
		S^c &= \{ f \in (\Lambda_{\infty})^{K[x]} \mid f(a) \neq v(a) \text{ for some } a \in K \} \\
		&= \bigcup_{a \in K} \{ f \in (\Lambda_{\infty})^{K[x]} \mid f(a) \neq v(a) \}.
	\end{aligned}
	\]
	
	For each $a \in K$, we have
	\[
	\{ f \in (\Lambda_{\infty})^{K[x]} \mid f(a) \neq v(a) \} = \pi_a^{-1}(\Lambda_{\infty} \setminus \{v(a)\}),
	\]
	where $\pi_a : (\Lambda_{\infty})^{K[x]} \longrightarrow \Lambda_{\infty}$ is the projection onto the $a$-th coordinate. 
	
	By property \textbf{(P2)}, the topology $\mathfrak{U}$ is Hausdorff, so the singleton $\{v(a)\}$ is closed in $\Lambda_{\infty}$. Hence, $\Lambda_{\infty} \setminus \{v(a)\}$ is open. Since $\pi_a$ is continuous, the preimage $\pi_a^{-1}(\Lambda_{\infty} \setminus \{v(a)\})$ is open in $(\Lambda_{\infty})^{K[x]}$. Therefore, $S^c$ is a union of open sets and thus open, which implies that $S$ is closed.
	
	Finally, by Theorem \ref{criterio}, $\mathcal{V}_{R}(\Lambda')$ is closed in $(\Lambda_{\infty})^{K[x]}$. Therefore, 
	\[
	\mathcal{T} = \mathcal{V}_{R}(\Lambda') \cap S
	\]
	is closed in $(\Lambda_{\infty})^{K[x]}$ with respect to the topology $\mathfrak{U}$.
\end{proof}

\begin{Cor}\label{CorTreeisClosed}
	The tree $\mathcal{T}$ is a closed set in $(\Lambda_{\infty})^{K[x]}$, when considering the order topology on $\Lambda_{\infty}$.
\end{Cor}
\begin{proof}
It follows immediately from Lemma~\ref{order} and Proposition~\ref{Tfechado}.
\end{proof}

\end{document}